\documentclass[11pt, reqno]{amsart}

\usepackage[T1]{fontenc}
\usepackage[latin1]{inputenc}

\usepackage{amsmath}
\usepackage{amsthm}
\usepackage{version}
\usepackage{xypic}
\usepackage{amssymb}

\usepackage[a4paper]{geometry}

\usepackage{mathtools}

\theoremstyle{plain}
\newtheorem{thm}{Theorem}[section]
\newtheorem{prop}[thm]{Proposition}
\newtheorem{lemma}[thm]{Lemma}
\newtheorem{coro}[thm]{Corollary}

\theoremstyle{definition}
\newtheorem{paragr}[thm]{}
\newtheorem{rem}[thm]{Remark}

\allowdisplaybreaks

\usepackage{xspace}

\makeatletter

\newcommand\DeclareMathOperatorRm[1]{%
  \expandafter\DeclareMathOperator\csname #1\endcsname{#1}}

\newcommand\DeclareMathOperatorSf[1]{%
  \expandafter\DeclareMathOperator\csname #1\endcsname{\mathsf{#1}}}

\newcommand\ndef[1]{\emph{#1}}

\def\xpoint{\futurelet\@let@token\@xpoint}
\def\@xpoint{%
  \ifx\@let@token.\else
    .%
  \fi
  \xspace}

\newcommand\zbox[1]{\makebox[0pt][l]{#1}}
\newcommand\pbox[1]{\zbox{\quad#1}}

\newcommand\C{\mathcal{C}}
\newcommand{\oo}{$\infty$\nobreakdash}
\newcommand\op{\mathrm{o}}
\newcommand\W{{\mathcal{W}}}
\newcommand\pref[1]{\widehat{#1}}
\newcommand\copref[1]{\widetilde{#1}}
\DeclareMathOperatorSf{Ob}
\DeclareMathOperatorSf{Fl}
\DeclareMathOperatorSf{Hom}
\DeclareMathOperatorRm{Aut}
\newcommand\id[1]{{1^{}_{#1}}}

\DeclareMathOperator{\Homi}{\underline{\mathsf{Hom}}}
\DeclareMathOperator{\Home}{\underline{\mathsf{Hom}}_{\,!}}
\DeclareMathOperator{\Homgl}{\underline{\mathsf{Hom}}_{\mathrm{\,gl}}}
\DeclareMathOperator{\Homglz}{\underline{\mathsf{Hom}}_{\mathrm{\,gl}_0}}

\DeclareMathOperator{\ExtGl}{\mathrm{Ext}_{\mathrm{gl}}}
\DeclareMathOperator{\ExtCat}{\mathrm{Ext}_{\mathrm{cat}}}
\DeclareMathOperator{\ExtGr}{\mathrm{Ext}_{\mathrm{gr}}}

\newcommand\Set{{\mathcal{S} \mspace{-2.mu}{et}}}
\newcommand{\Hot}{\mathsf{Hot}}
\newcommand{\Ho}[1]{\mathsf{Hot}_{#1}}

\newcommand{\@Cat}[3]{{\operatorname{\text{$#1$}\mathcal{#2C} \mspace{-2.mu}{at#3}}}}
\newcommand{\cat}{{\@Cat{}{}{}}}
\newcommand\Cat\cat
\newcommand{\ncat}[1]{{\@Cat{#1-}{}{}}}
\newcommand{\wcat}{{\@Cat{\infty-}{}{}}}

\newcommand{\@Grp}[3]{{%
  \operatorname{\text{$#1$}{\mathcal{#2G}\mspace{-0.5mu}rp\mspace{-0.01mu}d#3}}}}
\newcommand{\grp}{{\@Grp{}{}{}}}
\newcommand{\ngrp}[1]{{\@Grp{#1-}{}{}}}
\newcommand{\wgrp}{{\@Grp{\infty-}{}{}}}

\newcommand{\G}{\mathbb{G}}

\newcommand{\Thz}{\Theta_0}
\newcommand{\Th}{\Theta}
\newcommand{\Thtld}{\widetilde{\Theta}}

\newcommand\Mod[1]{{\operatorname{\mathrm{Mod}}(#1)}}

\newcommand\Ass[2]{(\mathrm{Ass}_{{#1},{#2}})}
\newcommand\Assx{(\mathrm{Ass})}
\newcommand\Exc[3]{(\mathrm{Exc}_{{#1},{#2},{#3}})}
\newcommand\Excx{(\mathrm{Exc})}
\newcommand\Lun[2]{(\mathrm{LUnit}_{{#1},{#2}})}
\newcommand\Lunx{(\mathrm{LUnit})}
\newcommand\Run[2]{(\mathrm{RUnit}_{{#1},{#2}})}
\newcommand\Runx{(\mathrm{RUnit})}
\newcommand\Fun[2]{(\mathrm{FUnit}_{{#1},{#2}})}
\newcommand\Funx{(\mathrm{FUnit})}
\newcommand\LInv[2]{(\mathrm{LInv}_{{#1},{#2}})}
\newcommand\LInvx{(\mathrm{LInv})}
\newcommand\RInv[2]{(\mathrm{RInv}_{{#1},{#2}})}
\newcommand\RInvx{(\mathrm{RInv})}
\newcommand\FInv[3]{(\mathrm{FInv}_{{#1},{#2},{#3}})}
\newcommand\FInvx{(\mathrm{FInv})}

\newcommand\DThtld{\mathcal{D}_{\Thtld}}
\newcommand\DThtldn[1]{\mathcal{D}_{\Thtld_{#1}}}
\newcommand\DTh{\mathcal{D}_{\Th}}

\newcommand\tabdimi[1][n]{%
{\left(
\begin{matrix}
i_1 && i_2 && \cdots && i_{#1} \cr
& i'_1 && i'_2 & \cdots & i'_{{#1}-1}
\end{matrix}
\right)}}

\newcommand\tabdim{{\tabdimi{}}}
\newcommand\tabdimijk{%
\left(
\begin{smallmatrix}
i & & j \cr
& k
\end{smallmatrix}
\right)}

\newcommand\at{\widetilde{a}}

\newcommand{\Dn}[1]{{\mathrm{D}_{#1}}}
\newcommand{\Dnt}[1]{{\widetilde{\mathrm{D}}_{#1}}}

\newcommand{\amalgd}[1]{\amalg^{}_{\Dn{#1}}}
\newcommand{\amalgdt}[1]{\amalg^{}_{\Dnt{#1}}}
\newcommand{\amalgX}[1]{\amalg^{}_{#1}}

\newcommand\Xt[1]{{\widetilde{X}_{#1}}}
\newcommand\fibrx[1]{\times^{}_{X_{#1}}}
\newcommand\fibrxt[1]{\times^{}_{\Xt{#1}}}

\newcommand\xb{\overline{x}}
\newcommand\yb{\overline{y}}
\newcommand\zb{\overline{z}}
\newcommand\tb{\overline{t}}

\newcommand{\Tht}[2][\@empty]{%
\ifx\@empty#1
\tau_{#2}^{}%
\else
\tau_{#1}^{#2}%
\fi
}
\newcommand{\Thtt}[2][\@empty]{%
\ifx\@empty#1
\widetilde{\tau}_{#2}^{}%
\else
\widetilde{\tau}_{#1}^{#2}%
\fi
}

\newcommand{\Ths}[2][\@empty]{%
\ifx\@empty#1
\sigma_{#2}^{}%
\else
\sigma_{#1}^{#2}%
\fi
}
\newcommand{\Thst}[2][\@empty]{%
\ifx\@empty#1
\widetilde{\sigma}_{#2}^{}%
\else
\widetilde{\sigma}_{#1}^{#2}%
\fi
}

\newcommand{\Thkp}[2][\@empty]{%
\ifx\@empty#1
\kappa'_{#2}%
\else
\kappa_{#1}^{'#2}%
\fi
}
\newcommand{\Thk}[2][\@empty]{%
\ifx\@empty#1
\kappa_{#2}^{}%
\else
\kappa_{#1}^{#2}%
\fi
}
\newcommand{\Thkt}[2][\@empty]{%
\ifx\@empty#1
\widetilde{\kappa}_{#2}^{}%
\else
\widetilde{\kappa}_{#1}^{#2}%
\fi
}
\newcommand{\Thnp}[2][\@empty]{%
\ifx\@empty#1
\nabla'_{#2}%
\else
\nabla_{#1}^{'#2}%
\fi
}
\newcommand{\Thn}[2][\@empty]{%
\ifx\@empty#1
\nabla_{#2}^{}%
\else
\nabla_{#1}^{#2}%
\fi
}
\newcommand{\Thnt}[2][\@empty]{%
\ifx\@empty#1
\widetilde{\nabla}_{#2}^{}%
\else
\widetilde{\nabla}_{#1}^{#2}%
\fi
}
\newcommand{\Thw}[2][\@empty]{%
\ifx\@empty#1
\Omega_{#2}^{}%
\else
\Omega_{#1}^{#2}%
\fi
}
\newcommand{\Thwt}[2][\@empty]{%
\ifx\@empty#1
\widetilde{\Omega}_{#2}^{}%
\else
\widetilde{\Omega}_{#1}^{#2}%
\fi
}

\newcommand\comp\ast
\newcommand\compt{\widetilde{\ast}\mspace{1mu}}
\newcommand\Glnto[2][\@empty]{%
\ifx\@empty#1
\widetilde{\ast}\mspace{1mu}_{#2}%
\else
    \widetilde{\ast}\mspace{1mu}_{#1}^{#2}%
  \fi
}
\newcommand\Glnt[2][\@empty]{%
  \ifx\@empty#1
    \,\widetilde{\ast}\mspace{1mu}_{#2}\,%
  \else
    \,\widetilde{\ast}\mspace{1mu}_{#1}^{#2}\,%
  \fi
}

\newcommand\Glp[1]{p^{}_{#1}}

\newcommand{\Glt}[2][\@empty]{%
  \ifx\@empty#1
    t_{#2}%
  \else
    t_{#1}^{#2}%
  \fi
}
\newcommand{\Gls}[2][\@empty]{%
  \ifx\@empty#1
    s_{#2}%
  \else
    s_{#1}^{#2}%
  \fi
}
\newcommand{\Gltt}[2][\@empty]{%
  \ifx\@empty#1
    \widetilde{t}_{#2}%
  \else
    \tilde{t}_{#1}^{#2}%
  \fi
}
\newcommand{\Glst}[2][\@empty]{%
  \ifx\@empty#1
    \widetilde{s}_{#2}%
  \else
    \tilde{s}_{#1}^{#2}%
  \fi
}
\newcommand{\Glw}[2][\@empty]{%
  \ifx\@empty#1
    w_{#2}^{}%
  \else
    w_{#1}^{#2}%
  \fi
}
\newcommand{\Glwt}[2][\@empty]{%
  \ifx\@empty#1
    \widetilde{w}\mspace{1mu}_{#2}^{}%
  \else
    \widetilde{w}\mspace{1mu}_{#1}^{#2}%
  \fi
}
\newcommand{\Glk}[2][\@empty]{%
  \ifx\@empty#1
    k_{#2}%
  \else
    k_{#1}^{#2}%
  \fi
}
\newcommand{\Glkt}[2][\@empty]{%
  \ifx\@empty#1
    \widetilde{k}_{#2}%
  \else
    \widetilde{k}_{#1}^{#2}%
  \fi
}
\newcommand{\Glkp}[2][\@empty]{%
  \ifx\@empty#1
    k'_{#2}%
  \else
    k_{#1}^{'#2}%
  \fi
}

\renewcommand\epsilon\varepsilon

\newcommand\cepsu{{\epsilon^{}_1}}

\newcommand\ceps[1]{{\epsilon^{}_{#1}}}
\newcommand\cepsp[1]{{\epsilon'_{#1}}}
\newcommand\cepsg[1]{\ceps{#1}}

\makeatother

\author{Dimitri Ara}
\address{Dimitri Ara, Institut Mathématiques de Jussieu, Université Paris
Diderot Paris 7, Case 7012, Bâtiment Chevaleret, 75205 Paris Cedex 13}
\email{ara@math.jussieu.fr}
\urladdr{http://people.math.jussieu.fr/~ara/}

 
\title[$\Thtld$ is a test category]{The groupoidal analogue $\Thtld$\\ to Joyal's category $\Th$ is a test category}

\thanks{
The author is grateful to Georges Maltsiniotis for having suggested him
this approach to prove that $\Thtld$ is a test category.
}

\begin{document}

\begin{abstract}
We introduce the groupoidal analogue $\Thtld$ to Joyal's cell category $\Th$ and
we prove that $\Thtld$ is a strict test category in the sense of Grothendieck.
This implies that presheaves on $\Thtld$ model homotopy types in a
canonical way. We also prove that the canonical functor from $\Th$ to
$\Thtld$ is aspherical, again in the sense of Grothendieck. This allows us
to compare weak equivalences of presheaves on $\Thtld$ to weak equivalences
of presheaves on $\Th$. Our proofs apply to other categories analogous to~$\Th$.
\end{abstract}

\maketitle

\section{Introduction}

In \emph{Pursuing Stacks}, Grothendieck defines a notion of weak
\oo-groupoid (see \cite{GrothPS}, \cite{MaltsiGr}, \cite{AraThese} and
\cite{MaltsiGrCat}) and conjectures that his weak \oo-groupoids model
homotopy types in a precise way. His definition of weak \oo-groupoids is
based on the notion of coherator. Roughly speaking, a coherator is a
category encoding the algebraic theory of weak \oo-groupoids. If $C$ is a
coherator, a weak \oo-groupoid (of type $C$) is a presheaf on
$C$ satisfying some left exactness condition. A first step toward
Grothendieck's conjecture would thus be to prove that presheaves on a
coherator (without the exactness condition) model homotopy types.

This is where test categories enter the picture. Test categories were
introduced by Grothendieck in \cite{GrothPS} (see also \cite{Maltsi} and
\cite{Cisinski}). The main property of these categories is that presheaves
on a test category canonically model homotopy types. Therefore, to prove
Grothendieck's conjecture, it is reasonable to start by trying to prove that
every coherator is a test category. In \cite{MaltsiGr}, Maltsiniotis
gave a series of conjectures implying Grothendieck's conjecture based on this
idea.

Conjecturally, every coherator is (non canonically) endowed with a
``forgetful'' functor to $\Thtld$. This is the reason why we are interested
in understanding the homotopy theory of $\Thtld$. In this article, we prove
that $\Thtld$ is a test category. Our hope is that we will be able to deduce
that every coherator $C$ is a test category from this result, using
properties of the functor from $C$ to $\Thtld$.

As announced in the title, $\Thtld$ is the groupoidal analogue to Joyal's
cell category~$\Th$. The category $\Th$ was introduced by Joyal in
\cite{JoyalTheta} as the opposite category of the category of so-called
finite disks.  Batanin and Street conjectured in \cite{BataninStreetTrees}
that $\Th$ could be seen as a full subcategory of the category of strict
\oo-categories. This was proved independently by Makkai and Zawadowski in
\cite{MakZaw} and by Berger in \cite{BergerNerve}. In the latter article,
Berger also gave a nice combinatorial description of $\Th$.  In our paper,
the category~$\Th$ is introduced as the universal categorical globular
extension. Roughly speaking, a categorical globular extension is a 
category endowed with operations dual to those of strict \oo-categories and
satisfying axioms dual to those of strict \oo-categories. 
We show, starting from our definition, that $\Th$ can be seen as the full
subcategory of the category of strict \oo-categories whose
objects are free strict \oo-categories on globular pasting schemes.
This implies, using the result of Berger, Makkai and Zawadowski, that our
category $\Th$ is canonically isomorphic to Joyal's cell category.

The category $\Thtld$ is defined in the same way by replacing strict
\oo-categories by strict \oo-groupoids. In our terminology, $\Thtld$ is the
universal groupoidal globular extension. We prove that $\Thtld$ can be seen
as the full subcategory of the category of strict \oo-groupoids whose
objects are free strict \oo-groupoids on globular pasting schemes. To our
knowledge, there is no known combinatorial description of $\Thtld$ analogous
to Berger's description of $\Th$.

In \cite{MaltsiTheta}, Cisinski and Maltsiniotis introduced the notion of
décalage and used it to prove that $\Th$ is a test category. They actually
proved that $\Th$ is a strict test category: a test category $A$ is strict if
the binary product of presheaves on $A$ is compatible with the product of homotopy types in a
strong sense. They proved that a small category satisfying some easy to check
condition, plus the existence of a splittable décalage, is a strict test
category. They then constructed a splittable décalage $\DTh$ on $\Th$ and
applied this
result to $\Th$. To construct this décalage, they used a beautiful description
of $\Th$ in terms of wreath products due to Berger (see
\cite{BergerWreath}). Another proof of the fact that $\Th$ is a test
category is given in Section 7.2 of the PhD thesis \cite{AraThese} of the
author.

Unfortunately, the category $\Thtld$ cannot be obtained using wreath
products and therefore the construction of the décalage $\DTh$ does not apply
to $\Thtld$. In this article, we construct a splittable décalage $\DThtld$ on
$\Thtld$ ``by hand''. The formulas defining $\DThtld$ are inspired by the
ones one can get by unfolding the definition of $\DTh$.
In particular, the construction of our décalage $\DThtld$ will also apply to
$\Th$ and, in this case, we will get the décalage $\DTh$ of \cite{MaltsiTheta}.

We deduce from the existence of the splittable décalage $\DThtld$ that $\Thtld$
is a strict test category. By a theorem of Cisinski, conjectured by
Grothendieck, this implies that the category of presheaves on $\Thtld$ is
endowed with a model category structure whose homotopy category is the
homotopy category of CW-complexes. There exists a canonical
functor from $\Th$ to $\Thtld$. Using the fact that this functor is
compatible with the décalages $\DTh$ and $\DThtld$, we deduce that it is
aspherical in the sense of Grothendieck. This implies that this functor
induces a Quillen equivalence between the Grothendieck-Cisinski model
category structures on presheaves on $\Th$ and on $\Thtld$. Note that the
Grothendieck-Cisinski model structure on presheaves on $\Th$ had already
been obtained by Berger in \cite{BergerNerve} using topological
techniques.

Moreover, our construction applies to other categories having similar
universal properties. For instance, the category $\Theta_{lr}$, which has a
universal property related to ``strict \oo-categories not necessarily
satisfying the axiom of functoriality of units'', is also a strict test
category and the canonical functor from $\Theta_{lr}$ to $\Th$ is
aspherical.

Most of the content of this article is extracted from the last chapter of
the PhD thesis \cite{AraThese} of the author. The calculations have been
entirely rewritten ``using elements'' (see Paragraph \ref{paragr:element}
for details).

Our paper is organized as follows. In Section 2, we recall the definitions of
strict \oo-categories and strict \oo-groupoids. We introduce the globular
language and in particular globular sums and globular extensions. We also
define the notion of categorical and groupoidal globular extensions, which
are in a sense dual to those of strict \oo-categories and strict
\oo-groupoids. In Section 3, we introduce the categories $\Thz$, $\Th$ and
$\Thtld$. In Section 4, we give a brief introduction to the theory of test
categories and we gather the definitions and results from \cite{MaltsiTheta}
about décalages that we will need.  We then enter the heart of the article.
In Section 5, we explain how to construct a new globular extension, the
twisted globular extension, from a globular extension endowed with some
comultiplications. In Section 6, we apply this construction to a groupoidal
globular extension and we show that the twisted globular extension is
endowed with a structure of groupoidal globular extension. In Section 7, we
use the results of Section 6 to build our décalage $\DThtld$. We show that
$\DThtld$ is splittable. In the final Section, we draw the consequences of
the previous Sections. We show that $\Thtld$ is a strict test category and
that the functor from $\Th$ to $\Thtld$ is aspherical. We explain how these
results generalize to other analogous categories.

If $C$ is a category, we will denote by $C^\op$ the opposite category. If
\[
\xymatrix@C=1pc@R=1pc{
X_1 \ar[dr]_{f_1} & & X_2 \ar[dl]^{g_1} \ar[dr]_{f_2} & &  \cdots & & X_n
\ar[dl]^{g_{n-1}} \\
& Y_1 & & Y_2 & \cdots & Y_{n-1}
}
\]
is a diagram in $C$, we will denote by 
\[ (X_1, f_1) \times_{Y_1} (g_1, X_2, f_2) \times_{Y_2} \dots
\times_{Y_{n-1}} (g_{n-1}, X_n) \]
its projective limit. Dually, we will denote by
\[ (X_1, f_1) \amalg_{Y_1} (g_1, X_2, f_2) \amalg_{Y_2} \dots
\amalg_{Y_{n-1}} (g_{n-1}, X_n) \]
the inductive limit of the corresponding diagram in $C^\op$.

\section{Strict $\infty$-categories and strict $\infty$-groupoids}

\begin{paragr}
We will denote by $\G$ the \ndef{globular category}, that is the category
generated by the graph
\[
\xymatrix{
\Dn{0} \ar@<.6ex>[r]^-{\Ths{1}} \ar@<-.6ex>[r]_-{\Tht{1}} &
\Dn{1} \ar@<.6ex>[r]^-{\Ths{2}} \ar@<-.6ex>[r]_-{\Tht{2}} &
\cdots \ar@<.6ex>[r]^-{\Ths{i-1}} \ar@<-.6ex>[r]_-{\Tht{i-1}} &
\Dn{i-1} \ar@<.6ex>[r]^-{\Ths{i}} \ar@<-.6ex>[r]_-{\Tht{i}} &
\Dn{i} \ar@<.6ex>[r]^-{\Ths{i+1}} \ar@<-.6ex>[r]_-{\Tht{i+1}} &
\dots
}
\]
and the coglobular relations
\[\Ths{i+1}\Ths{i} = \Tht{i+1}\Ths{i}\quad\text{and}\quad\Ths{i+1}\Tht{i} =
\Tht{i+1}\Tht{i}, \qquad i \ge 1.\]
For $i \ge j \ge 0$, we will denote by $\Ths[j]{i}$ and $\Tht[j]{i}$ the
morphisms from $\Dn{j}$ to $\Dn{i}$ defined by
\[\Ths[j]{i} = \Ths{i}\cdots\Ths{j+2}\Ths{j+1}\quad\text{and}\quad
  \Tht[j]{i} = \Tht{i}\cdots\Tht{j+2}\Tht{j+1}.\]

A \ndef{globular set} or \ndef{\oo-graph} is a presheaf on $\G$. The datum of
a globular set $X$ amounts to the datum of a diagram of sets
\[
\xymatrix{
\cdots \ar@<.6ex>[r]^-{\Gls{i+1}} \ar@<-.6ex>[r]_-{\Glt{i+1}} &
X_{i} \ar@<.6ex>[r]^-{\Gls{i}} \ar@<-.6ex>[r]_-{\Glt{i}} &
X_{i-1} \ar@<.6ex>[r]^-{\Gls{i-1}} \ar@<-.6ex>[r]_-{\Glt{i-1}} &
\cdots \ar@<.6ex>[r]^-{\Gls{2}} \ar@<-.6ex>[r]_-{\Glt{2}} &
X_1 \ar@<.6ex>[r]^-{\Gls{1}} \ar@<-.6ex>[r]_-{\Glt{1}} &
X_0
}
\]
satisfying the globular relations
\[\Gls{i}\Gls{i+1} = \Gls{i}\Glt{i+1}\quad\text{and}\quad\Glt{i}\Gls{i+1} =
\Glt{i}\Glt{i+1}, \qquad i \ge 1.\]
For $i \ge j \ge 0$, we will denote by $\Gls[j]{i}$ and $\Glt[j]{i}$ the
maps from $X_i$ to $X_j$ defined by
\[\Gls[j]{i} = \Gls{j+1}\cdots\Gls{i-1}\Gls{i}\quad\text{and}\quad
  \Glt[j]{i} = \Glt{j+1}\cdots\Glt{i-1}\Glt{i}.\]

A \ndef{morphism of globular sets} is a morphism of presheaves on $\G$.
\end{paragr}

\begin{paragr}
An \ndef{\oo-precategory} is a globular set $X$ endowed with maps
\[
  \begin{split}
  \comp_j^i & : \Big(X_i, \Gls[j]{i}\big) \times_{X_j} (\Glt[j]{i}, X_i) \to 
      X_i,\quad i > j \ge 0,
      \\
  \Glk{i} & : X_i \to X_{i+1}, \quad i \ge 0,
  \end{split}
\]
such that
\begin{enumerate}
    \item 
      for every 
      $(u, v)$ in $(X_i, \Gls[j]{i}) \times_{X_j} (\Glt[j]{i}, X_i)$ with
      $i > j \ge 0$, we have
  \[
  \Gls{i}\big(u \comp_j^i v\big) = 
  \begin{cases}
    \Gls{i}(v), & j = i - 1, \\
    \Gls{i}(u) \comp_j^{i-1} \Gls{i}(v), & j < i - 1,
  \end{cases}
  \]
  and
\[
  \Glt{i}(u \comp_j^i v) = 
  \begin{cases}
    \Glt{i}(u), & j = i - 1, \\
    \Glt{i}(u) \comp_j^{i-1} \Glt{i}(v), & j < i - 1;
  \end{cases}
  \]
  \item for every $u$ in $X_i$ with $i \ge 0$, we have
  \[
  \Gls{i+1}\Glk{i}(u) = u = \Glt{i+1}\Glk{i}(u).
  \] 
\end{enumerate}

For $i \ge j \ge 0$, we will denote by $\Glk[i]{j}$ the map from $X_j \to
X_i$ defined by
  \[ \Glk[i]{j} = \Glk{i-1}\cdots\Glk{j+1}\Glk{j}. \]

A \ndef{morphism of \oo-precategories} is a morphism of globular sets
between \oo-pre\-cat\-e\-go\-ries which is compatible with the $\ast^i_j$'s and the
$\Glk{i}$'s in an obvious way.

An \oo-precategory $X$ is a \ndef{strict \oo-category}
if it satisfies the following axioms:
\begin{itemize}
  \item $\Ass{i}{j},\quad i > j \ge 0$,\\
     for every $(u, v, w)$ in
    $(X_i, \Gls[j]{i}) \times_{X_j} (\Glt[j]{i}, X_i,
    \Gls[j]{i}) \times_{X_j} (\Glt[j]{i}, X_i)$,
    we have
    \[ (u \comp^i_j v) \comp^i_j w = u \comp^i_j (v \comp^i_j w)\text{;} \]
  \item $\Exc{i}{j}{k}, \quad i > j > k \ge 0$,\\
   for every $(u, u', v, v')$ in
    \[ (X_i, \Gls[j]{i}) \times_{X_j} (\Glt[j]{i}, X_i, \Gls[k]{i})
    \times_{X_k} (\Glt[k]{i}, X_i, \Gls[j]{i}) \times_{X_j} (\Glt[j]{i},
    X_i),\]
    we have
    \[ (u \comp^i_j u') \comp^i_k (v \comp^i_j v') = (u \comp^i_k v)
    \comp^i_j ( u' \comp^i_k v')\text{;} \]
  \item $\Lun{i}{j}, \quad i > j \ge 0$, \\
   for every $u$ in $X_i$, we have 
       \[ \Glk[i]{j}\Glt[j]{i}(u) \comp^i_j u  = u\text{;} \]
  \item $\Run{i}{j}, \quad i > j \ge 0$, \\
   for every $u$ in $X_i$, we have 
       \[ u \comp^i_j \Glk[i]{j}\Gls[j]{i}(u) = u\text{;} \]
  \item $\Fun{i}{j}, \quad i > j \ge 0$, \\
    for every 
    $(u, v)$ in $(X_i, \Gls[j]{i}) \times_{X_j} (\Glt[j]{i} ,X_i)$,
    we have
    \[ \Glk{i}(u \comp^i_j v) = \Glk{i}(u) \comp^{i+1}_j \Glk{i}(v). \]
\end{itemize}

The \ndef{category of strict \oo-categories} is the full subcategory of the
category of \oo-pre\-categories whose objects are strict \oo-categories.
We will denote it by $\wcat$.
\end{paragr}

\begin{paragr}\label{paragr:def_oo-grpd}
An \ndef{\oo-pregroupoid} $X$ is an \oo-precategory endowed with maps
\[
  \Glw[j]{i} : X_i \to X_i, \quad i > j \ge 0,
\]
such that
for every $u$ in $X_i$ for $ i \ge 1$ and $j$ such that $i > j \ge
0$, we have
\[
\Gls{i}(\Glw[j]{i}(u)) = 
\begin{cases}
\Glt{i}(u), & j = i - 1,\\
\Glw[j]{i-1}(\Gls{i}(u)), & j < i - 1,
\end{cases}
\]
and
\[
\Glt{i}(\Glw[j]{i}(u)) = 
\begin{cases}
\Gls{i}(u), & j = i - 1,\\
\Glw[j]{i-1}(\Glt{i}(u)), & j < i - 1.
\end{cases}
\]

A \ndef{morphism of \oo-pregroupoids} is a morphism of \oo-precategories
between \oo-group\-oids which is compatible with the $\Glw[j]{i}$'s in an
obvious way.

An \oo-pregroupoid $X$ is a \ndef{strict \oo-groupoid} if it is a
strict \oo-category and if it satisfies the following axioms:
\begin{itemize}
\item $\LInv{i}{j}, \quad i > j \ge 0$,\\
for every $u$ in $X_i$, we have
\[
\Glw[j]{i}(u) \comp^i_j u = \Glk[i]{j}(\Gls[j]{i}(u));
\]
\item $\RInv{i}{j}, \quad i > j \ge 0$,\\
for every $u$ in $X_i$, we have
\[
u \comp^i_j \Glw[j]{i}(u) = \Glk[i]{j}(\Glt[j]{i}(u)).
\]
\end{itemize}

The \ndef{category of strict \oo-groupoids} is the full subcategory of the
category of \oo-pre\-group\-oids whose objects are strict \oo-groupoids.
We will denote it by $\wgrp$. Note that a morphism of strict \oo-categories
between strict \oo-groupoids is automatically
a morphism of strict \oo-groupoids.

Although it is not clear from our definition, being a strict \oo-groupoid is
a property of a strict \oo-category. More precisely, if $X$ is a
strict \oo-category such that for every $i > j \ge 0$, every $i$-arrow $u$ admits
a $\comp^i_j$-inverse (i.e., an $i$-arrow $\Glw[j]{i}(u)$ satisfying the
axioms $\LInv{i}{j}$ and $\RInv{i}{j}$), then $X$ is endowed with a unique
structure of \oo-groupoid.  Note also that our axioms for strict
\oo-groupoids are highly redundant. For instance, for a strict \oo-category
to be a strict \oo-groupoids, it suffices to ask for $\comp^i_0$-inverses or
for $\comp^i_{i-1}$-inverses for every $i \ge 1$ (see Proposition 2.3 of
  \cite{AraMetWGrpd}).

One can easily check that a strict \oo-groupoid automatically satisfies
the following additional axiom:
\begin{itemize}
\item $\FInv{i}{j}{j'}, \quad i > j, j' \ge 0$,\\
for every $(u, v)$ in $(X_i, \Gls[j]{i}) \times_{X_j} (\Glt[j]{i} ,X_i)$,
we have
\[
\Glw[j']{i}(u \comp^i_j v) =
\begin{cases}
\Glw[j']{i}(v) \comp^i_{j} \Glw[j']{i}(u), & j = j', \\
\Glw[j']{i}(u) \comp^i_{j} \Glw[j']{i}(v), & j \neq j'.
\end{cases}
\]
\end{itemize}
More precisely, if an \oo-pregroupoid satisfies Axioms $\Assx$,
$\Excx$, $\Lunx$, $\Runx$ and $\RInvx$, then it satisfies
Axiom $\FInvx$ (where by the name of an axiom without subscripts, we denote
the conjunction on all meaningful subscripts of this axiom).
\end{paragr}

\section{The categories $\Thz$, $\Th$ and $\Thtld$}

\begin{paragr}
Let $n$ be a positive integer. A \ndef{table of dimensions} of width
$n$ is the datum of integers $i_1, \dots, i_n, i'_1, \dots, i'_{n-1}$ such
that
\[ i_k > i'_k\quad\text{and}\quad i_{k+1} > i'_k, \qquad 1 \le k \le n - 1. \]
We will denote such a table of dimensions by
\[
\tabdim.
\]

Let $(C, F)$ be a category under $\G$, that is a category $C$ endowed with a functor
$F : \G \to C$. We will denote in the same way the objects and morphisms
of $\G$ and their image by the functor $F$. Let
\[ T = \tabdim \]
be a table of dimensions. The \ndef{globular sum} in $C$ associated to $T$
(if it exists) is the iterated amalgamated sum
\[ (\Dn{i_1}, \Ths[i'_1]{i_1}) \amalgd{i'_1} (\Tht[i'_1]{i_2}, \Dn{i_2},
\Ths[i'_2]{i_2}) \amalgd{i'_2} \dots
\amalgd{i'_{n-1}} (\Tht[i'_{n-1}]{i_n}, \Dn{i_n}) \]
in $C$.
We will denote it briefly by
\[
\Dn{i_1} \amalgd{i'_1} \Dn{i_2} \amalgd{i'_2} \dots
\amalgd{i'_{n-1}} \Dn{i_n}.
\]
If the table of dimensions $T$ is understood, for $k$ such that $1 \le k \le n$, we will denote
by $\cepsg{k}$ the canonical morphism
\[ \cepsg{k} : \Dn{i_k} \to \Dn{i_1} \amalgd{i'_1} \Dn{i_2} \amalgd{i'_2} \dots
\amalgd{i'_{n-1}} \Dn{i_n}. \]
The pair $(C, F)$ is said to be a \ndef{globular extension} if
for every table of dimensions $T$ (of any width), the globular sum
associated to $T$ exists in $C$. We will often say, by abuse of language,
that $C$ is a globular extension. 

Let $C$ and $D$ be two globular extensions. A \ndef{morphism of globular
extensions} from $C$ to $D$ is a functor from $C$ to $D$ under $\G$
(that is such that the triangle
\[
\xymatrix@C=1pc@R=1pc{
& \G \ar[ld] \ar[rd] \\
C \ar[rr] & & D
}
\]
is commutative) which respects globular sums. We will also call such a functor a
\ndef{globular functor}. 

If $C$ is a category under $\G$ and $D$ is a category,
we will denote by $\Homgl(C, D)$ the full subcategory of the category
$\Homi(C,D)$ of functors from $C$ to $D$, whose objects are functors
$C \to D$ such that $(D, \G \to C \to D)$ is a globular extension.
In particular, when $C = \G$ (and $\G \to \G$ is the identity functor),
the objects of $\Homgl(\G, D)$ are the globular extension structures on $D$.
We will also denote this category by $\ExtGl(D)$.

Let $C$ be a globular extension. A \ndef{model} of $C$ or \ndef{globular
presheaf} on $C$ is a presheaf $G : C^\op \to \Set$ which respects
globular products (i.e., limits dual to globular sums). We will denote by
$\Mod{C}$ the full subcategory of the category $\pref{C}$ of presheaves on
$C$ whose objects are globular presheaves.
\end{paragr}

\begin{prop}[Universal property of $\Thz$]\label{prop:prop_univ_thz}
There exists a globular extension $\Thz$ such that for every category
$C$, the precomposition by the functor $\G \to \Thz$ induces an equivalence
of categories 
\[
\Homgl(\Thz, C) \to \ExtGl(C).
\]
Moreover, for every such $\Thz$, this equivalence of categories is
surjective on objects.
\end{prop}

\begin{proof}
Consider the Yoneda functor $\G \to \pref{\G}$. For each table of
dimensions $T$, choose a globular sum $S_T$ in $\pref{\G}$. Let $\Thz$ be the
full subcategory of $\pref{\G}$ whose objects are the $S_T$'s.

Before proving that $\Thz$ has the desired universal property, let us
introduce some notations. If $A$ is a category, we will denote by
$\copref{A}$ the category of copresheaves on $A$, that is the category
$\Homi(A, \Set)^\op$. If $B$ is a second category, we will denote by
$\Home(A, B)$ the full subcategory of $\Homi(A, B)$ whose objects are
functors preserving inductive limits.

Let $C$ be a category. We will construct a quasi-inverse to the canonical
functor
\[ U : \Homgl(\Thz, C) \to \ExtGl(C). \]
Let
\[ U' : \Home(\pref{\G}, \copref{C}) \to \Homi(\G, \copref{C}) \]
be the functor induced by the Yoneda functor $\G \to \pref{\G}$.
Since the category $\copref{C}$ is cocomplete, the universal property of
$\pref{\G}$ gives us a quasi-inverse $L'$ of $U'$. 
Consider now the functor $G$ defined by the composition
\[
\ExtGl(C) \to \Homi(\G, \copref{C}) \xrightarrow{L'} \Home(\pref{\G}, \copref{C})
\to \Homi(\Thz, \copref{C}),
\]
where the first and the last functors are respectively induced by the
(contravariant) Yoneda functor $C \to \copref{C}$ and the inclusion $\Thz
\to \pref{\G}$.  Since the Yoneda functor $C \to \copref{C}$ preserves
inductive limits, the functor $G$ factors through $\Homgl(\Thz, C)$ and
gives rise to a functor
\[ L : \ExtGl(C) \to \Homgl(\Thz, C). \]
One easily checks that $L$ is a quasi-inverse of $U$.

Since the second assertion is invariant under equivalences of categories
under $\G$, it suffices to prove it for the category $\Thz$ we have just
built. The assertion then follows from the fact that the functor $U'$ is
surjective on objects.
\end{proof}

\begin{paragr}
Two globular extensions satisfying the above universal property
are uniquely equivalent up to a unique natural isomorphism. 
One can show that the objects of such a globular extension have no
automorphisms. In particular, a skeletal version of such a globular
extension (i.e., such that isomorphic objects are equal) is
unique up to a unique isomorphism. We will denote by $\Thz$ this globular
extension.

Note that the above universal property states in particular that $\Thz$ is
the free globular completion of $\G$ in the following sense:
if $(C, F : \G \to C)$ is a globular extension, there exists a globular
functor $\Thz \to C$ unique up to a unique natural isomorphism. More
precisely, the choice of such a functor $\Thz \to C$ amounts to the choice
of a globular sum for every table of dimensions.

The category $\Thz$ defined above is canonically isomorphic to the category
$\Thz$ defined in terms of finite planar rooted trees by Berger in
\cite{BergerNerve}. Berger's definition is explained in detail in Section
2.3 of \cite{AraThese}. See also Section 4 of \cite{WeberGen} or
\cite{MakZaw} for a description of the bijection between tables of dimensions
and finite planar rooted trees. Note that tables of dimensions are called
zig-zag sequences in \cite{WeberGen} and ud-vectors (standing for up and
down vectors) in \cite{MakZaw}.
\end{paragr}

\begin{paragr}
Let $C$ be a globular extension. If $X$ is a globular presheaf on $C$, then
by restricting it to $\G$, we obtain a globular set.
We thus have a canonical functor
\[ \Mod{C} \to \pref{\G}. \]
\end{paragr}

\begin{prop}\label{prop:mod_thz}
The functor
\[
\Mod{\Thz} \to \pref{\G}
\]
is an equivalence of categories.
\end{prop}

\begin{proof}
This is exactly what the universal property of $\Thz$ claims when applied to
$\Set^\op$.
\end{proof}

\begin{paragr}
If $C$ is a globular extension, the Yoneda functor
$C \to \pref{C}$ factors through $\Mod{C}$. We thus have a functor
$C \to \Mod{C}$.

By the previous proposition, the functor 
\[
\Thz \to \Mod{\Thz} \to \pref{\G}
\]
is fully faithful. A globular set which is in the image of this functor will
be called a \ndef{globular pasting scheme}. We can thus view $\Thz$ as the
full subcategory of $\pref{\G}$ whose objects are globular pasting schemes.

Note that in the bijection between tables of dimensions and finite planar
rooted trees, the above functor from $\Thz$ to $\pref{\G}$ associates to a
tree $T$ the globular set $T^\ast$ introduced by Batanin in
\cite{BataninWCat}. The globular pasting schemes can also be characterized
as the cardinals (in the sense of Definition 4.16 of \cite{WeberFam})
of the free strict \oo-category functor $\pref{\G} \to \wcat$ (see
Section 9 of \cite{WeberGen}).
\end{paragr}

\begin{paragr}
A \ndef{globular extension under $\Thz$} is a category $C$ endowed with a
functor $\Thz \to C$ such that $(C, \G \to \Thz \to C)$ is a globular
extension.  If $C$ is a globular extension under $\Thz$, the globular sum
associated to a table of dimensions is uniquely defined. A \ndef{morphism of
globular extensions under $\Thz$} is a functor under $\Thz$ between globular
extensions under $\Thz$. Note that such a functor automatically respects
globular sums.

If $C$ is a category under $\Thz$ and $D$ is a category,
we will denote by $\Homglz(C, D)$ the full subcategory of the category
$\Homi(C,D)$ whose objects are functors $C \to D$ such that $(D, \Thz \to C
\to D)$ is a globular extension under $\Thz$.
\end{paragr}

\begin{prop}
Let $C$ be a category under $\Thz$. There exists a globular extension
$\overline{C}$ under $\Thz$, endowed with a functor $C \to \overline{C}$
under $\Thz$ such that the functor $C \to \overline{C}$ induces an
isomorphism of categories
\[ \Homglz(\overline{C}, D) \to \Homglz(C, D). \]
\end{prop}

\begin{proof}
This is a special case of a standard categorical construction (see
Proposition 3 of \cite{Ehresmann}). See also Section 2.6 of \cite{AraThese}
and Paragraph 3.10 of \cite{MaltsiGrCat}.
\end{proof}

\begin{paragr}
If $C$ is a category under $\Thz$, the globular extension $\overline{C}$ of
the previous proposition (which is unique up to a unique isomorphism) will
be called the \ndef{globular completion} of~$C$. Note that the functor $C
\to \overline{C}$ is bijective on objects.
\end{paragr}

\begin{paragr}\label{paragr:precat}
A \ndef{precategorical globular extension} is a globular extension $C$ under
$\Thz$ endowed with morphisms
\[
\begin{split}
\Thn[j]{i} & : \Dn{i} \to \Dn{i} \amalgd{j} \Dn{i},\qquad i > j \ge 0, \\
\Thk{i} & : \Dn{i+1} \to \Dn{i},\qquad i \ge 0,
\end{split}
\]
such that
\begin{enumerate}
\item\label{item:sbg_n} for every $i, j$ such that $i > j \ge 0$, we have
\[ 
\Thn[j]{i}\Ths{i} =
\begin{cases}
    \epsilon^{}_2\Ths{i}, & j = i - 1, \\
    \big(\Ths{i} \amalgd{j} \Ths{i}\big)\Thn[j]{i-1}
    & j < i - 1,
\end{cases}
\]
and
\[
\Thn[j]{i}\Tht{i} =
\begin{cases}
    \epsilon^{}_1\Tht{i}, & j = i - 1, \\
    \big(\Tht{i} \amalgd{j} \Tht{i}\big)\Thn[j]{i-1}
    & j < i - 1,
\end{cases}
\]
where $\epsilon^{}_1,\epsilon^{}_2 : \Dn{i} \to \Dn{i}\amalgd{i-1} \Dn{i}$
denote the canonical morphisms;
\item\label{item:sbg_k} for every $i \ge 0$, we have
\[ \Thk{i}\Ths{i+1} = \id{\Dn{i}} \quad\text{and}\quad \Thk{i}\Tht{i+1} =
\id{\Dn{i}}\text{.} \]
\end{enumerate}

If $C$ is a precategorical globular extension, for $i \ge j \ge 0$, we will denote
by $\Thk[i]{j}$ the morphism from $\Dn{i}$ to $\Dn{j}$ defined by
\[ \Thk[i]{j} = \Thk{j}\dots\Thk{i-2}\Thk{i-1}, \]
and, for $i > 0$, we set
\[ \Thn{i} = \Thn[i-1]{i}. \]

A \ndef{morphism of precategorical globular extensions} is a morphism of
globular extensions under $\Thz$ between precategorical globular extensions
preserving the $\Thn[j]{i}$'s and the $\Thk{i}$'s.

A precategorical globular extension is \ndef{categorical} if it satisfies
the following axioms:
  \begin{itemize}
    \item $\Ass{i}{j}, \quad i > j \ge 0$,\\
      the following square commutes:
      \[
      \xymatrixrowsep{3pc}
      \xymatrixcolsep{4.5pc}
       \xymatrix{
     \Dn{i} \ar[r]^{\Thn[j]{i}} \ar[d]_{\Thn[j]{i}} &
      \Dn{i} \amalgd{j} \Dn{i} \ar[d]^{\id{\Dn{i}} \amalgd{j} \Thn[j]{i}}\\
      \Dn{i} \amalgd{j} \Dn{i} \ar[r]_-{\Thn[j]{i} \amalgd{j}
      \id{\Dn{i}}} &
      \Dn{i} \amalgd{j} \Dn{i} \amalgd{j} \Dn{i} \pbox{;}
       }
      \]
    \item $\Exc{i}{j}{k}, \quad i > j > k \ge 0$,\\
      the following diagram commutes:
      \[
      \xymatrixrowsep{3pc}
      \xymatrixcolsep{-.3pc}
      \xymatrix{
      & \Dn{i} \ar[ld]_{\Thn[k]{i}} \ar[rd]^{\Thn[j]{i}} \\
      \Dn{i} \amalgd{k} \Dn{i}\ar[d]_{\Thn[j]{i} \amalgd{k} \Thn[j]{i}}& &
      \Dn{i} \amalgd{j} \Dn{i} \ar[d]^{\Thn[k]{i}
      \amalgX{\Thn[k]{j}} \Thn[k]{i}}\\
      (\Dn{i} \amalgd{j} \Dn{i}) \amalgd{k} (\Dn{i} \amalgd{j}
      \Dn{i}) \ar[rr]^-{\sim} & &
      (\Dn{i} \amalgd{k} \Dn{i}) \amalgX{\Dn{j} \amalgd{k} \Dn{j}}
      (\Dn{i} \amalgd{k} \Dn{i}),
      }
      \]
      where the left amalgamated sum is
      \[
      (\Dn{i} \amalgd{k} \Dn{i},
      \Ths[j]{i} \amalgd{k} \Ths[j]{i}) \amalgX{\Dn{j} \amalgd{k} \Dn{j}}
      (\Tht[j]{i} \amalgd{k} \Tht[j]{i}, \Dn{i} \amalgd{k}
      \Dn{i})\pbox{;}
      \]
      
    \item $\Lun{i}{j},\quad i > j \ge 0$,\\
      the following triangle commutes:
   \[
      \xymatrixrowsep{3pc}
      \xymatrixcolsep{4pc}
       \xymatrix{
       \Dn{i} \ar[d]_{\Thn[j]{i}} \ar[dr]^{\sim} \\
      \Dn{i} \amalgd{j} \Dn{i} 
      \ar[r]_{\Thk[i]{j} \amalgd{j} \id{\Dn{i}}} &
      \Dn{j} \amalgd{j} \Dn{i} \pbox{;}
      }
      \]
    \item $\Run{i}{j},\quad i > j \ge 0$,\\
      the following triangle commutes:
     \[
      \xymatrixrowsep{3pc}
      \xymatrixcolsep{4pc}
       \xymatrix{
       \Dn{i} \ar[d]_{\Thn[j]{i}} \ar[dr]^{\sim} \\
      \Dn{i} \amalgd{j} \Dn{i} 
      \ar[r]_{\id{\Dn{i}} \amalgd{j} \Thk[i]{j}}&
      \Dn{i} \amalgd{j} \Dn{j} \pbox{;}
      }
      \]
    \item $\Fun{i}{j},\quad i > j \ge 0$,\\
      the following square commutes:
      \[
      \xymatrixrowsep{3pc}
      \xymatrixcolsep{4.5pc}
      \xymatrix{
      \Dn{i+1} \ar[r]^-{\Thn[j]{i+1}} \ar[d]_{\Thk{i}} & 
      \Dn{i+1}\amalgd{j} \Dn{i+1} \ar[d]^{\Thk{i} \amalgd{j}
      \Thk{i}} \\
      \Dn{i} \ar[r]_-{\Thn[j]{i}} &\Dn{i}\amalgd{j} \Dn{i} \pbox{.}
       }
      \]
  \end{itemize}

The \ndef{category of categorical globular extensions} is the full
subcategory of the category of precategorical globular extensions whose
objects are categorical globular extensions.

If $C$ is a category, we will denote by $\ExtCat(C)$ the category whose
objects are categorical globular extension structures on $C$, i.e., functors
from $\Thz$ to $C$ endowed with $\Thn[j]{i}$'s and $\Thk{i}$'s making $C$ a
categorical globular extension, and whose morphisms are natural transformations.
\end{paragr}

\begin{prop}[Universal property of $\Th$]\label{prop:prop_univ_th}
There exists a categorical globular extension $\Th$ such
that for every category $C$, 
the precomposition by the functor $\Thz \to \Th$ induces an isomorphism of categories
\[
\Homglz(\Th, C) \to \ExtCat(C).
\]
\end{prop}

\begin{proof}
\newcommand\Thpcat{\Theta_{\mathrm{pcat}}}
Let $\Thpcat$ be the globular completion of the category obtained from
$\Thz$ by formally adjoining morphisms $\Thk{i}$ and $\Thn[j]{i}$
satisfying the relations of precategorical globular extensions.

Let now $\Th$ be the globular completion of the category obtained from
$\Thpcat$ by formally imposing the commutativity of the diagrams appearing
in the definition of categorical globular extensions.

It is clear that $\Th$ has the desired universal property.
\end{proof}

\begin{paragr}
We will denote by $\Th$ the categorical globular extension of the previous
proposition (which is unique up to a unique isomorphism). 
We will see that this category is canonically isomorphic to Joyal's cell
category introduced in \cite{JoyalTheta}.
Note that the functor $\Thz \to \Th$ is bijective on objects. 
\end{paragr}

\begin{paragr}
Let $C$ be a categorical globular extension.  If $X$ is a globular presheaf
on $C$, the globular set obtained by restricting $X$ to $\G$ is
canonically endowed with a structure of strict \oo-category whose compositions are
the
\[ \comp_j^i = X(\Thn[j]{i}) : X_i \times_{X_j} X_i \to X_i, \quad i > j \ge 0, \]
and whose units are the
\[ \Glk{i} = X(\Thk{i}) : X_i \to X_{i+1}, \quad i > 0. \]
We thus have a canonical functor
\[ \Mod{C} \to \wcat. \]
\end{paragr}

\begin{prop}\label{prop:mod_th}
The functor
\[ \Mod{\Th} \to \wcat \]
is an equivalence of categories.
\end{prop}

\begin{proof}
This is an immediate consequence of the universal property of $\Th$ applied to
$\Set^\op$ and of Proposition \ref{prop:mod_thz}.
\end{proof}

\begin{prop}\label{prop:desc_th}
The functor
\[ \Th  \to \Mod{\Th} \to \wcat \]
identifies $\Th$ with the full subcategory of $\wcat$ whose objects are free
strict \oo-categories on globular pasting schemes.
\end{prop}

\begin{proof}
By the previous proposition, this functor is fully faithful. It thus
suffices to describe its image.

If $C$ is a globular extension, we will denote by
\[ C \xrightarrow{i^{}_C} \Mod{C} \xrightarrow{j^{}_C} \pref{C} \]
the canonical decomposition of the Yoneda functor. 
By Propositions 1.27 and 1.51 of~\cite{Adamek}, the functor $i^{}_C$ admits a
left adjoint $r^{}_C$.

Let now $u : C \to D$ be a morphism of globular extensions.
Denote by $u^\ast : \pref{D} \to \pref{C}$ the restriction functor 
and by $u_! : \pref{C} \to \pref{D}$ its left adjoint. The functor $u^\ast$
induces a functor $u^\bullet : \Mod{D} \to \Mod{C}$. Moreover, this functor
admits $u_\bullet = r_Du_!j^{}_C$ as a left adjoint and the square
\[
\xymatrix{
C \ar[r]^u \ar[d]_{i^{}_C} & D \ar[d]^{i^{}_D} \\
\Mod{C} \ar[r]_{u_\bullet} & \Mod{D}
}
\]
is commutative up to isomorphism. In particular, 
if $k : \Thz \to \Th$ denotes the canonical morphism,
the square
\[
\xymatrix{
\Thz \ar[d]_{i^{}_{\Thz}} \ar[r]^k & \Th \ar[d]^{i^{}_{\Th}} \\
\Mod{\Thz} \ar[r]_{k_\bullet} & \Mod{\Th} \\
}
\]
is commutative up to isomorphism.

Let $U$ be the forgetful functor $\wcat \to \pref{\G}$ and let $L : \pref{\G} \to
\wcat$ be its left adjoint, i.e., the free strict \oo-category functor.
The square
\[
\xymatrix{
\Mod{\Thz} \ar[d] & \Mod{\Th} \ar[l]_(.45){k^\bullet} \ar[d] \\
\pref{\G} & \wcat \ar[l]^U \pbox{,}
}
\]
where the vertical functors are the equivalences of categories
of Propositions \ref{prop:mod_thz} and \ref{prop:mod_th}, is obviously
commutative. It follows that the square
\[
\xymatrix{
\Mod{\Thz} \ar[r]^{k_\bullet} \ar[d] & \Mod{\Th} \ar[d] \\
\pref{\G} \ar[r]_-L & \wcat
}
\]
is commutative up to isomorphism.

We thus obtain that the diagram
\[
\xymatrix{
\Thz \ar[d]_{i^{}_{\Thz}} \ar[r]^k & \Th \ar[d]^{i^{}_{\Th}} \\
\Mod{\Thz} \ar[r]^{k_\bullet} \ar[d] & \Mod{\Th} \ar[d] \\
\pref{\G} \ar[r]^-L & \wcat
}
\]
is commutative up to isomorphism, hence the result.
\end{proof}

\begin{prop}
The category $\Th$ is canonically isomorphic to Joyal's cell category.
\end{prop}

\begin{proof}
By Theorem 5.10 of \cite{MakZaw} (or Theorem 1.12 of \cite{BergerNerve}),
Joyal's cell category is canonically isomorphic to the full subcategory of
$\wcat$ described in the previous proposition. Hence the result by this
proposition.
\end{proof}

\begin{paragr}\label{paragr:pregr}
A \ndef{pregroupoidal globular extension} is a precategorical globular
extension endowed with morphisms
\[
\Thw[j]{i} : \Dn{i} \to \Dn{i}, \qquad i > j \ge 0,
\]
such that for all $i, j$ satisfying $i > j \ge 0$, we have
\[
\Thw[j]{i}\Ths{i} =
\begin{cases}
\Tht{i} & j = i - 1,\\
\Ths{i}\Thw[j]{i-1} & j < i - 1,
\end{cases}
\]
and
\[
\Thw[j]{i}\Tht{i} =
\begin{cases}
\Ths{i} & j = i - 1,\\
\Tht{i}\Thw[j]{i-1} & j < i - 1.
\end{cases}
\]


A \ndef{morphism of pregroupoidal globular extensions} is a morphism of
precategorical globular extensions between pregroupoidal globular extensions
preserving the $\Thw[j]{i}$'s.

A pregroupoidal globular extension is \ndef{groupoidal} if it is categorical
and if it satisfies the following additional axioms:
\begin{itemize}
\item $\LInv{i}{j}, \quad i > j \ge 0$, \\
the following square commutes:
\[
\xymatrixrowsep{3pc}
\xymatrixcolsep{3.5pc}
\xymatrix{
  \Dn{i} \ar[d]_{\Thn[j]{i}} \ar[r]^{\Thk[i]{j}} 
  &  \Dn{j} \ar[d]^{\Ths[j]{i}}
   \\
  \Dn{i} \amalgd{j} \Dn{i} \ar[r]_-{(\Thw[j]{i}, \id{\Dn{i}})} & \Dn{i}
  \pbox{;}
}
\]
\item $\RInv{i}{j}, \quad i > j \ge 0$, \\
the following square commutes:
\[
\xymatrixrowsep{3pc}
\xymatrixcolsep{3.5pc}
\xymatrix{
  \Dn{i} \ar[d]_{\Thn[j]{i}} \ar[r]^{\Thk[i]{j}} 
  &  \Dn{j} \ar[d]^{\Tht[j]{i}}
   \\
  \Dn{i} \amalgd{j} \Dn{i} \ar[r]_-{(\id{\Dn{i}},\Thw[j]{i})} & \Dn{i}
  \pbox{.}
}
\]
\end{itemize}

The \ndef{category of groupoidal globular extensions} is the full
subcategory of the category of pregroupoidal globular extensions whose
objects are groupoidal globular extensions.

If $C$ is a category, we will denote by $\ExtGr(C)$ the category whose
objects are groupoidal globular extension structures on $C$, i.e., functors
from $\Thz$ to $C$ endowed with $\Thn[j]{i}$'s, $\Thk{i}$'s and $\Thw[j]{i}$'s
making $C$ a groupoidal globular extension, and whose morphisms are natural
transformations.
\end{paragr}

\begin{prop}[Universal property of $\Thtld$]\label{prop:prop_univ_thtld}
There exists a groupoidal globular extension $\Thtld$ such that for every
category $C$, the precomposition by the functor $\Thz \to \Thtld$ induces an
isomorphism of categories
\[
\Homglz(\Thtld, C) \to \ExtGr(C).
\]
\end{prop}

\begin{proof}
The proof is similar to the one of the categorical case
(Proposition \ref{prop:prop_univ_th}).
\end{proof}

\begin{paragr}
We will denote by $\Thtld$ the groupoidal globular extension of the previous
proposition (which is unique up to a unique isomorphism). The category
$\Thtld$ is the groupoidal analogue to Joyal's category $\Th$. Note that the
functor $\Thz \to \Thtld$ is bijective on objects. 
\end{paragr}

\begin{paragr}
Let $C$ be a groupoidal globular extension. As in the categorical case, if
$X$ is a globular presheaf on $C$, the globular set obtained by restricting $X$
to $\G$ is canonically endowed with a structure of strict \oo-category, and
this \oo-category is a strict \oo-groupoid whose inverses are given by the
\[ \Glw[j]{i} = X(\Thw[j]{i}) : X_i \to X_i, \quad i > j \ge 0. \]
We thus have a canonical functor
\[ \Mod{C} \to \wgrp. \]
\end{paragr}

The two following propositions are proved exactly as in the categorical
case.

\begin{prop}\label{prop:mod_thtld}
The functor
\[ \Mod{\Thtld} \to \wgrp \]
is an equivalence of categories.
\end{prop}

\begin{prop}\label{prop:desc_thtld}
The functor
\[ \Thtld  \to \Mod{\Thtld} \to \wgrp \]
identifies $\Thtld$ with the full subcategory of $\wgrp$ whose objects are free
strict \oo-group\-oids on globular pasting schemes.
\end{prop}

\section{Test categories and décalages}

\begin{paragr}
We recall that if $A$ is a small category, we denote by $\pref{A}$ the
category of presheaves on $A$. Let $u : A \to B$ be a functor and
$b$ an object of $B$. We will denote by $A/b$ the comma category
whose objects are pairs $(a, f : u(a) \to b)$ where $a$ is an object of $A$
and $f$ a morphism of $B$, and whose morphisms from an object $(a, f)$ to an
object $(a', f')$ are morphisms $g : a \to a'$ of $A$ such that $f'u(g) = f$.
In particular, if $A$ is a small category and $F$ is a presheaf on $A$, the
category $A/F$ (where $u : A \to \pref{A}$ is the Yoneda functor) is the
category of elements of $F$.
\end{paragr}

\begin{paragr}
We will denote by $\Cat$ the category of small categories. We recall that a
\ndef{weak equivalence} of small categories is a functor which is sent by
the nerve functor on a weak equivalence of simplicial sets. We will denote by
$\W$ the class of weak equivalences of small categories and by $\Hot$ the
Gabriel-Zisman localization $\Cat[\W^{-1}]$ of $\Cat$ by $\W$. A famous
theorem of Quillen states that $\Hot$ is canonically equivalent to the
homotopy category of simplicial sets (see Corollary 3.3.1 of \cite{Illusie})
and hence to the homotopy category of CW-complexes.
\end{paragr}

\begin{paragr}
Let $A$ be a small category. We have a pair of adjoint functors
\[
\begin{matrix}
 i^{}_A : & \pref{A}  & \to & \Cat & \qquad &
i^*_A : & \Cat &  \to & \pref{A}\hfill\\
 & F &  \mapsto &  A/F & \qquad &
 & C & \mapsto & \big(a \mapsto \Hom_\Cat(A/a, C)\big).
\end{matrix}
\]
A morphism of presheaves on $A$ is a \ndef{weak equivalence} if it is sent
by $i^{}_A$ on a weak equivalence of small categories. We will denote by
$\W_{\pref{A}}$ the class of weak equivalences of presheaves on $A$ and by
$\Ho{A}$ the Gabriel-Zisman localization of $\pref{A}$ by $\W_{\pref{A}}$.
The functor $i^{}_A$ induces a functor $\overline{i^{}_A} : \Ho{A} \to
\Hot$.  If $i^*_A(\W) \subset \W_{\pref{A}}$, i.e., if $i^{}_Ai^*_A(\W) \subset
\W$, then the functor $i^*_A$ induces a functor $\overline{i^*_A} : \Hot \to
\Ho{A}$. Moreover, if this condition is satisfied, the pair of adjoint
functors $(i^{}_A, i^*_A)$ induces a pair of adjoint functors
$(\overline{i^{}_A}, \overline{i^*_A})$.
\end{paragr}

\begin{paragr}
A small category $A$ is a \ndef{weak test category} if the
following conditions are satisfied:
\begin{itemize}
\item we have $i^*_A(\W) \subset \W_{\pref{A}}$ ;
\item for every presheaf $F$ on $A$, the unit morphism $\eta^{}_F : F \to
i^*_Ai^{}_A(F)$ belongs to $\W_{\pref{A}}$ ;
\item for every small category $C$, the counit morphism $\epsilon^{}_C :
i^{}_Ai^*_A(C) \to C$ belongs to $\W$.
\end{itemize}
The two last conditions are the obvious sufficient conditions for the
adjunction $(\overline{i^{}_A}, \overline{i^*_A})$ to be an equivalence
adjunction. In particular, if $A$ is a weak test category, the category
$\Ho{A}$ is canonically equivalent to $\Hot$.
\end{paragr}

\begin{paragr}
A small category $A$ is a \ndef{local test category} if for
every object $a$ of $A$, the category $A/a$ is a weak test category. A
small category is a \ndef{test category} if it is a weak test category and a
local test category.

A test category $A$ is a \ndef{strict test category} if the functor
$i^{}_A$ respects binary products up to weak equivalence, i.e., if
for all presheaves $F$ and $G$ on $A$, the canonical functor
\[A/(F \times G) \to A/F \times A/G\]
is a weak equivalence.
\end{paragr}

\begin{thm}[Grothendieck-Cisinski]\label{thm:test_cm}
Let $A$ be a local test category. Then $(\pref{A}, \W_{\pref{A}})$ is endowed with a
structure of model category whose cofibrations are the monomorphisms. This
model category structure is cofibrantly generated and proper.

Moreover, if $A$ is a strict test category, weak equivalences are stable
by binary products.
\end{thm}

\begin{proof}
See Corollary 4.2.18 of \cite{Cisinski} for the model category structure.
The properness follows by Theorem 4.3.24 of \cite{Cisinski} and by the case of
simplicial sets.

The last assertion is obvious.
\end{proof}

\begin{paragr}
A small category $A$ is \ndef{aspherical} if the unique functor from $A$ to
the terminal category is a weak equivalence. It is easy to check that
categories admitting a terminal object are aspherical.  One can prove (see
Remark 1.5.4 of \cite{Maltsi}) that a local test category is test if and
only if it is aspherical. We will only need the following obvious case: a
local test category with a terminal object is a test category.

Let $u : A \to B$ be a functor between small categories. The functor $u$ is
\ndef{aspherical} if for every object $b$ of $B$, the category $A/b$ is
aspherical. 

Let $A$ be a small category. A presheaf $F$ on a $A$ is \ndef{aspherical} if
the category $A/F$ is aspherical. Every representable presheaf is aspherical
since for every object $a$ of $A$, the category $A/a$ admits a terminal
object.

If $u : A \to B$ is a functor between small categories, we will denote by
$u^\ast : \pref{B} \to \pref{A}$ the restriction functor and by $u_\ast :
\pref{A} \to \pref{B}$ its right adjoint.
\end{paragr}

\begin{prop}
Let $u : A \to B$ be a functor between aspherical small categories. The
following properties are equivalent:
\begin{enumerate}
\item the functor $u$ is aspherical;
\item for every morphism $\varphi : F \to G$ of presheaves on $B$, the
morphism $\varphi$ is a weak equivalence of presheaves on $B$ if and only if
the morphism $u^\ast(\varphi)$ is a weak equivalence of presheaves on $A$.
\end{enumerate}
\end{prop}

\begin{proof}
See \cite{GrothPS} or Proposition 1.2.9 of \cite{Maltsi}.
\end{proof}

\begin{prop}\label{prop:asp}
Let $u : A \to B$ be an aspherical functor between test categories. Then 
$(u^\ast, u_\ast)$ is a Quillen equivalence (where $\pref{A}$ and $\pref{B}$
are endowed with the Grothendieck-Cisinski model structure of Theorem
\ref{thm:test_cm}).
\end{prop}

\begin{proof}
See Proposition 4.2.24 of \cite{Cisinski}.
\end{proof}

\begin{paragr}
Let $A$ be a small category. Denote by $\varnothing_{\pref{A}}$ the initial
presheaf on $A$ and by $e_{\pref{A}}$ the terminal one. An \ndef{interval}
$(I, \partial_0, \partial_1)$ on $\pref{A}$ consists of a presheaf $I$ on
$A$ and two morphisms $\partial_0, \partial_1 : e_{\pref{A}} \to F$. Such an
interval is \ndef{separating} if the equalizer of $\partial_0$ and
$\partial_1$ is $\varnothing_{\pref{A}}$.
\end{paragr}

\begin{paragr}
Let $A$ be a small category. A \ndef{décalage} on $A$ consists of
an endofunctor $D : A \to A$, an object $a_0$ of $A$ and two natural
transformations
\[
\xymatrix{
  \id{A} \ar[r]^\alpha & D & a_0 \ar[l]_\beta
}
\]
(where $a_0$ denotes the constant endofunctor whose value is $a_0$).
We will denote by $(A, a_0, D, \alpha, \beta)$ such a décalage.
A \ndef{splitting} of $(A, a_0, D, \alpha, \beta)$
consists of a retraction $r_a : D(a) \to a$ of $\alpha^{}_a$ for
every object $a$ of~$A$.  Note that the $r_a$'s are \emph{not} asked to be
functorial in $a$. A décalage is \ndef{splittable} if it admits a splitting.
\end{paragr}

\begin{prop}\label{prop:crit_test}
Let $A$ be a small category. If $A$ admits a splittable décalage and $\pref{A}$
admits a separating interval $(I, \partial_0, \partial_1)$ such that $I$ is
aspherical, then $A$ is a strict test category.
\end{prop}

\begin{proof}
See Proposition 3.6 and Corollary 3.7 of \cite{MaltsiTheta}.
\end{proof}

\begin{paragr}
Let $\mathcal{D}_A = (A, a_0, D, \alpha, \beta)$ and $\mathcal{D}_B = (B, b_0,
E, \gamma, \delta)$ be two décalages. A \ndef{morphism of décalages} from $\mathcal{D}_A$
to $\mathcal{D}_B$ is a functor $u : A \to B$ such that
\[ uD = Eu,\quad u(a_0) = b_0,\quad u \ast \alpha = \gamma \ast u, \quad\text{and}\quad u \ast
\beta = \delta \ast u.\]
\end{paragr}

\begin{prop}\label{prop:shift_asp}
Let $u : A \to B$ be a functor between small categories. If there exists a
décalage $\mathcal{D}_A$ on $A$ and a splittable décalage $\mathcal{D}_B$ on $B$ such
that $u$ induces a morphism of décalages from $\mathcal{D}_A$ to $\mathcal{D}_B$,
then $u$ is aspherical.
\end{prop}

\begin{proof}
See Proposition 3.9 of \cite{MaltsiTheta}.
\end{proof}

\section{Shifted globular extensions}\label{sec:shifted_ge}

\begin{paragr}
In this section, we fix a globular extension $(C, F)$
endowed with morphisms
\[ \Thn{i} : \Dn{i} \to \Dn{i} \amalgd{i-1} \Dn{i}, \quad i \ge 1, \]
such that
\[ 
\Thn{i}\Ths{i} = \epsilon^{}_2\Ths{i}
\quad\text{and}\quad
\Thn{i}\Tht{i} = \epsilon^{}_1\Tht{i},
\]
where $\epsilon^{}_1,\epsilon^{}_2 : \Dn{i} \to \Dn{i}\amalgd{i-1} \Dn{i}$
denote the canonical morphisms.

The purpose of the section is to define a new structure of globular
extension on $C$,  i.e., a functor $K : \G \to C$ such that
$(C, K)$ is a globular extension, using the $\Thn{i}$'s. We will call $(C,
K)$ the \ndef{twisted globular extension} of $(C, F)$ (by the $\Thn{i}$'s).
\end{paragr}

\begin{paragr}
We set
\[
\Dnt{i} = \Dn{1} \amalgd{0} \Dn{2} \amalgd{1} \dots
\amalgd{i-1} \Dn{i+1}, \quad i \ge 1.
\]
Recall that we denote the canonical morphisms by
\[\ceps{k} : \Dn{k} \to \Dnt{i},\quad 1 \le k \le i + 1.\]
We define morphisms
\[
\begin{split}
\Thst{i} & : \Dnt{i-1} \to \Dnt{i}, \quad i \ge 1,\\
\Thtt{i} & : \Dnt{i-1} \to \Dnt{i}, \quad i \ge 1,
\end{split}
\]
by the formulas
\[
\begin{split}
\Thst{i} & =
(\ceps{1}, \dots, \ceps{i-1}, (\ceps{i}, \ceps{i+1}\Tht{i+1})\Thn{i}),
\\
\Thtt{i} & =
(\ceps{1}, \dots, \ceps{i}).
\end{split}
\]
(It is obvious that $\Thtt{i}$ is well-defined and we will prove that
$\Thst{i}$ is well-defined in Paragraph \ref{paragr:element}.)

Let $X$ be a globular presheaf on $C$. We set
\[
\Xt{i} = X(\Dnt{i}) =
X_{1} \fibrx{0} X_{2} \fibrx{1} \dots \fibrx{i-1} X_{i+1},
\quad i \ge 1.
\]
For $k$ such that $1 \le k \le i + 1$, we will denote by
$\Glp{k}$ the canonical projection
\[ \Glp{k} : \Xt{i} \to X_k. \]
We will often denote by $\xb$ an element of $\Xt{i}$ and by $x_1, \dots,
x_{i+1}$ the components of $\xb$.

We define maps
\[
\begin{split}
\Glst{i} & : \Xt{i} \to \Xt{i-1}, \quad i \ge 1, \\
\Gltt{i} & : \Xt{i} \to \Xt{i-1}, \quad i \ge 1,
\end{split}
\]
by the formulas dual to the ones defining $\Thst{i}$ and $\Thtt{i}$:
\[
\begin{split}
\Glst{i}(x_1, \dots, x_{i+1}) & = (x_1, \dots, x_{i-1}, x_i \comp_{i-1}^i
\Glt{i+1}(x_{i+1})), \\
\Gltt{i}(x_1, \dots, x_{i+1}) & = (x_1, \dots, x_i).
\end{split}
\]
In particular, once we have proved that $\Thst{i}$ is well-defined, we
will have
\[
\Glst{i} = X(\Thst{i})
\quad\text{and}\quad
\Gltt{i} = X(\Thtt{i}).
\]

\end{paragr}

\begin{paragr}\label{paragr:element}
Let $i \ge 1$. Let us prove that $\Thst{i}$ is well-defined. We need to
show that
\[
\ceps{i-1}\Ths{i-1} = (\ceps{i}, \ceps{i+1}\Tht{i+1})\Thn{i}\Tht{i}\Tht{i-1}.
\]
But
\[
\begin{split}
(\ceps{i}, \ceps{i+1}\Tht{i+1})\Thn{i}\Tht{i}\Tht{i-1}
& =
(\ceps{i}, \ceps{i+1}\Tht{i+1})\cepsu\Tht{i}\Tht{i-1}
\\
& =
\ceps{i}\Tht{i}\Tht{i-1} \\
& =
\ceps{i-1}\Ths{i-1}.
\end{split}
\]
This calculation was straightforward. However, in the sequel of this paper,
we will need to prove more and more complicated identities. For this reason,
we will prove our identities ``using elements''. In \cite{AraThese}, we gave
proofs without using this technique. The result is barely readable.

Let us explain what we mean by ``using elements''. Let $f, g : S \to T$ be
two parallel morphisms of $C$. Suppose we want to prove that $f$ is equal to
$g$. By the (contravariant) Yoneda lemma, it suffices to check
that for every object $U$ of $C$, the two maps
\[
 \Hom_C(T, U) \to \Hom_C(S, U),
\]
induced by $f$ and $g$, are equal. Since every representable presheaf on $C$ is
globular, it suffices to prove
that for every globular presheaf $X$ on $C$, the two maps
\[
 \Hom_{\pref{C}}(T, X) \to \Hom_{\pref{C}}(S, X),
\]
induced by $f$ and $g$, are equal. But by the Yoneda lemma, these maps
correspond to the maps
\[
X(f), X(g) : X(T) \to X(S).
\]
In conclusion, the morphisms $f$ and $g$ are equal if and only the
maps $X(f)$ and $X(g)$ are equal for every globular presheaf $X$.

Let us apply this to
\[
f = \ceps{i-1}\Ths{i-1}
\quad\text{and}\quad
g = (\ceps{i}, \ceps{i+1}\Tht{i+1})\Thn{i}\Tht{i}\Tht{i-1}.
\]
Let $X$ be a globular presheaf on $C$.
For $\xb$ in $\Xt{i}$, we have
\[
X(f)(\xb) = \Gls{i-1}(x_{i-1})
\quad\text{and}\quad
X(g)(\xb) = \Glt{i-1}\Glt{i}(x_i \comp_{i-1}^i \Glt{i+1}(x_{i+1})).
\]
But
\[
\Glt{i-1}\Glt{i}(x_i \comp_{i-1}^i \Glt{i+1}(x_{i+1})) =
\Glt{i-1}\Glt{i}(x_i) = \Gls{i-1}(x_{i-1}). \]
We have thus given another proof of the well-definedness of $\Thst{i}$.
\end{paragr}

\emph{From now on, we fix a globular presheaf $X$ on $C$.}

\begin{prop}\label{prop:def_K}
The maps
\[
\Dn{i}  \mapsto \Dnt{i}, \quad
\Ths{i+1}  \mapsto \Thst{i+1}, \quad 
\Tht{i+1}  \mapsto \Thtt{i+1}, \quad i \ge 0,
\]
define a functor $\G \to C$.
\end{prop}

In the sequel of this section, we will denote this functor by $K$.

\begin{proof}
We need to prove that the $\Thst{i}$'s and $\Thtt{i}$'s satisfy the
coglobular relations. By Paragraph \ref{paragr:element}, it suffices to show that
the $\Glst{i}$'s and $\Gltt{i}$'s satisfy the globular relations.

Let $i \ge 2$ and $\xb$ in $\Xt{i}$. We have 
\[
\begin{split}
\Glst{i-1}\Glst{i}(\xb)
& =
\Glst{i-1}(x_1, \dots, x_{i-1}, x_i \comp_{i-1}^i \Glt{i+1}(x_{i+1})) \\
& =
(x_1, \dots, x_{i-2}, x_{i-1} \comp_{i-2}^{i-1}
\Glt{i}(x_i \comp_{i-1}^i \Glt{i+1}(x_{i+1})))\\
& =
(x_1, \dots, x_{i-2}, x_{i-1} \comp_{i-2}^{i-1}
\Glt{i}(x_i)) \\
& =
\Glst{i-1}(x_1, \dots, x_i) \\
& =
\Glst{i-1}\Gltt{i}(\xb)
\end{split}
\]
and
\[
\begin{split}
\Gltt{i-1}\Glst{i}(\xb)
& =
\Gltt{i-1}(x_1, \dots, x_{i-1}, x_i \comp_{i-1}^i \Glt{i+1}(x_{i+1})) \\
& =
(x_1, \dots, x_{i-1})\\
& =
\Gltt{i-1}\Gltt{i}(\xb),
\end{split}
\]
hence the result.
\end{proof}

We collect in the following lemma two identities related to the
structure of $\Xt{i}$ that we will use several times.

\begin{lemma} \label{lemma:struct_xt}
Let $\xb$ in $\Xt{i}$. We have
\[
\begin{split}
\Gls[l]{l+2}(x_{l+2}) & = \Gls[l]{i+1}(x_{i+1}), \quad 0 \le l \le i-1, \\
\Gls{l+1}(x_{l+1}) & = \Glt[l]{i+1}(x_{i+1}), \quad 0 \le l \le i-1.
\end{split}
\]
\end{lemma}

\begin{proof}
We have
\[
\begin{split}
\Gls[l]{l+2}(x_{l+2})
& = \Gls{l+1}\Gls{l+2}(x_{l+2}) \\
& = \Gls{l+1}\Glt{l+2}\Glt{l+3}(x_{l+3}) \\
& = \Gls[l]{l+3}(x_{l+3}) \\
& = \cdots \\
& = \Gls[l]{i+1}(x_{i+1}), \\
\end{split}
\]
and
\[
\begin{split}
\Gls{l+1}(x_{l+1})
& = \Glt[l]{l+2}(x_{l+2}) \\
& = \Glt{l+1}\Glt{l+2}(x_{l+2}) \\
& = \Glt{l+1}\Gls{l+2}(x_{l+2}) \\
& = \Glt{l+1}\Glt[l+1]{l+3}(x_{l+3}) \\
& = \Glt[l]{l+3}(x_{l+3}) \\
& = \cdots \\
& = \Glt[l]{i+1}(x_{i+1}).
\end{split}
\]
\end{proof}

\begin{paragr}
Let us introduce some more notations. We set
\[
\begin{split}
\Dnt{j,i} & = \Dn{j+1} \amalgd{j} \Dn{j+2} \amalgd{j+1} \dots
\amalgd{i-1} \Dn{i+1}, \quad i \ge j \ge 0.
\end{split}
 \]
In particular, we have
\[
\begin{split}
\Dnt{0,i} & = \Dnt{i},\quad i \ge 0,\\
\Dnt{i} & = \Dnt{0,k} \amalgd{k} \Dnt{k+1,i},
\quad i > k \ge 0, \\
\Dnt{j,i} & = \Dnt{j,k} \amalgd{k}
\Dnt{k+1,i}, \quad i > k \ge j \ge 0.
\end{split}
\]
Dually, we set
\[
\begin{split}
\Xt{j,i} & = X(\Dnt{j, i}) = X_{j+1} \fibrx{j} X_{j+2} \fibrx{j+1} \dots
\fibrx{i-1} X_{i+1}, \quad i \ge j \ge 0,
\end{split}
 \]
and we have
\[
\begin{split}
\Xt{0,i} & = \Xt{i},\quad i \ge 0,\\
\Xt{i} & = \Xt{0,k} \fibrx{k} \Xt{k+1,i},
\quad i > k \ge 0, \\
\Xt{j,i} & = \Xt{j,k} \fibrx{k}
\Xt{k+1,i}, \quad i > k \ge j \ge 0.
\end{split}
\]
\end{paragr}

We will now prove that $(C, K)$ is a globular extension.

\begin{lemma}\label{lemma:cocart}
Let $\C$ be a category and let $f : X \to Y$, $g^{}_X : A \to X$, $g^{}_Y :
A \to X$ and $g^{}_Z : A \to Z$ be morphisms of $\C$. Suppose that the
amalgamated sums
\[
X \amalg_{A} Z = (X, g^{}_X) \amalg_{A} (g_Z, Z)
\quad\text{and}\quad
Y \amalg_{A} Z = (Y, g^{}_Y) \amalg_{A} (g_Z, Z)
\]
exist in $\C$, and that we have $fg^{}_X = g^{}_Y$, so that the morphism
\[ X \amalg_{A} Z \xrightarrow{f\, \amalg_A Z} Y \amalg_{A} Z \]
is well-defined. Then the square
\[
\xymatrix{
  X \ar[d]_{f} \ar[r] & X \amalg_{A} Z \ar[d]^{f\, \amalg_A Z} \\
  Y \ar[r] & Y \amalg_{A} Z \\
}
\]
is cocartesian.
\end{lemma}

\begin{proof}
The square of the statement is the coproduct in $A\backslash\C$ of the
squares
\[
\xymatrix{
  X \ar[r]^{\id{X}} \ar[d]_f & X \ar[d]|{}="a"^f &
  &
  A \ar[r]^{g^{}_Z} \ar[d]|{}="b"_{\id{A}} & Z
  \ar[d]^{\id{Z}} \\ 
  \ar@{}"a";"b"|{\textstyle\text{and}}
  Y \ar[r]_{\id{Y}} & Y &
  &
  A \ar[r]_{g^{}_Z} & Z\pbox{,} \\
}
\]
which are both cocartesian in $A\backslash\C$.
\end{proof}

\begin{prop}
Let $T  = \tabdimijk$ be a table of dimensions of width $2$. The globular
sum $\Dnt{i} \amalgdt{k} \Dnt{j}$ associated to $T$ in $(C, K)$ exists and
is canonically isomorphic to
\[
\Dnt{i} \amalgd{k} \Dnt{k+1,j} =
(\Dnt{i}, \ceps{i+1}\Ths[k]{i+1}) \amalgd{k}
(\cepsg{1}\Tht[k]{k+2}, \Dnt{k+1,j}).
\]
\end{prop}

\begin{proof}
We prove that the square
\[
\xymatrix{
\Dnt{k} \ar[r] \ar[d]_{\Thst[k]{i}} &
\Dnt{j} = \Dnt{k} \amalgd{k} \Dnt{k+1, j} 
\ar[d]^{\Thst[k]{i} \amalgd{k} \Dnt{k+1,j}} \\
\Dnt{i} \ar[r] &
(\Dnt{i}, \ceps{i+1}\Ths[k]{i+1}) \amalgd{k}
(\cepsg{1}\Tht[k]{k+2}, \Dnt{k+1,j})
}
\]
is cocartesian by applying the previous lemma to
\[
X = \Dnt{k},\quad Y = \Dnt{i},\quad Z = \Dnt{k+1,j},\quad A = \Dn{k},
\]
and
\[
f = \Thst[k]{i}, \quad g^{}_X = \ceps{k+1}\Ths{k+1}, \quad g^{}_Y = \ceps{i+1}\Ths[k]{i+1}, 
\quad g^{}_Z = \cepsg{1}\Tht[k]{k+2}.
\]
The two amalgamated sums appearing in the square exist since they are globular
sums in $(C, F)$. Hence, to apply the lemma, it
suffices to check that
\[ \Thst[k]{i}\ceps{k+1}\Ths{k+1} = \ceps{i+1}\Ths[k]{i+1}. \]
Let us prove this identity using elements. Let $\xb$ in $\Xt{i}$. We need to
prove that
\[ \Gls{k+1}\Glp{k+1}\Glst[k]{i}(\xb) = \Gls[k]{i+1}(x_{i+1}). \]
But
\[
\begin{split}
\Gls{k+1}\Glp{k+1}\Glst[k]{i}(\xb) & =
\Gls{k+1}\Glp{k+1}\Glst{k+1}\Gltt[k+1]{i}(\xb) \\
& =
\Gls{k+1}\Glp{k+1}\Glst{k+1}(x_1, \dots, x_{k+2}) \\
& =
\Gls{k+1}(x_{k+1} \comp^{k+1}_k
\Glt{k+2}(x_{k+2})) \\
& = \Gls{k+1}\Glt{k+2}(x_{k+2}) \\
& = \Gls[k]{k+2}(x_{k+2}) \\
& = \Gls[k]{i+1}(x_{i+1}), \\
\end{split}
\]
where the last equality follows from Lemma \ref{lemma:struct_xt}.
\end{proof}

\begin{prop}\label{prop:shifted_ge}
Let 
\[ T = \tabdim \]
be a table of dimensions. The globular sum
$\Dnt{i_1} \amalgdt{i'_1} \dots \amalgdt{i'_{n-1}} \Dnt{i_n}$
associated to $T$ in $(C, K)$
 exists
and is canonically isomorphic to
\[
\Dnt{i_1} \amalgd{i'_1} \Dnt{i'_1+1, i_2} \amalgd{i'_2} \Dnt{i'_2+1, i_3}
\amalgd{i'_3} \dots
\amalgd{i'_{n-1}} \Dnt{i'_{n-1} + 1, i_n}.
\]

In particular, $(C, K)$ is a globular extension.
\end{prop}

As announced at the beginning of this section, we will call $(C, K)$ the
\ndef{twisted globular extension} of $(C, F)$ (by the $\Thn{i}$'s).

\begin{proof}
We prove the result by induction on the width $n$ of the table of
dimensions. Suppose
\[
\begin{split}
\MoveEqLeft \Dnt{i_2} \amalgdt{i'_2} \dots \amalgdt{i'_{n-1}} \Dnt{i_n} \\
& = 
\Dnt{i_2} \amalgd{i'_2} \Dnt{i'_2+1, i_3} \amalgd{i'_3} \dots
\amalgd{i'_{n-1}} \Dnt{i'_{n-1} + 1, i_n} \\
& =
\Dnt{i'_1} \amalgd{i'_1}
\Big(
\Dnt{i'_1+1, i_2} \amalgd{i'_2} \Dnt{i'_2+1, i_3} \amalgd{i'_3} \dots
\amalgd{i'_{n-1}} \Dnt{i'_{n-1} + 1, i_n}
\Big).
\end{split}
\]

As in the proof of the previous proposition, by using Lemma
\ref{lemma:cocart}, we obtain that the square
\[
\xymatrix{
\Dnt{i'_1} \ar[r] \ar[d]_{\Thst[i'_1]{i_1}} &
\Dnt{i'_1} \amalgd{i'_1}
\Big(
\Dnt{i'_1+1, i_2} \amalgd{i'_2}
\amalgd{i'_2} \Dnt{i'_2+1, i_3} \dots
\amalgd{i'_{n-1}} \Dnt{i'_{n-1} + 1, i_n}
\Big)
\ar[d]^{\Thst[i'_1]{i_1}\amalgd{i'_1} \id{}}
\\
\Dnt{i_1} \ar[r] &
\Dnt{i_1} \amalgd{i'_1}
\Big(
\Dnt{i'_1+1, i_2} \amalgd{i'_2}
\amalgd{i'_2} \Dnt{i'_2+1, i_3} \dots
\amalgd{i'_{n-1}} \Dnt{i'_{n-1} + 1, i_n}
\Big)
}
\]
is cocartesian. Hence the result.
\end{proof}

\begin{paragr}\label{paragr:iso_can}
Dually, if
\[ \tabdim \]
is a table of dimensions, the globular product
$\Xt{i_1} \fibrxt{i'_1} \dots \fibrxt{i'_{n-1}} \Xt{i_n}$ exists
and is canonically isomorphic to 
\[
\Xt{i_1} \fibrx{i'_1} \Xt{i'_1+1, i_2} \fibrx{i'_2} \Xt{i'_2+1, i_3}
\fibrx{i'_3} \dots
\fibrx{i'_{n-1}} \Xt{i'_{n-1} + 1, i_n}.
\]
Moreover, the canonical isomorphism
\[
c : \Xt{i_1} \fibrxt{i'_1} \dots  \fibrxt{i'_{n-1}} \Xt{i_n} \to
\Xt{i_1} \fibrx{i'_1} \Xt{i'_1+1,i_2} \fibrx{i'_2} \dots 
\fibrx{i'_{n-1}} \Xt{i'_{n-1}+1, i_n}
\]
is given by the formula
\[
\begin{split}
\MoveEqLeft
c\big(x^1_1, \dots, x^1_{i_1 + 1}, x^2_1, \dots, x^2_{i_2 + 1},
\dots, x^n_1, \dots, x^n_{i_n + 1}\big) \\
& =
\big(x^1_1, \dots, x^1_{i_1 + 1}, x^2_{i'_1+2}, \dots, x^2_{i_2 + 1},
\dots, x^n_{i'_{n-1} + 2}, \dots, x^n_{i_n + 1}\big). \\
\end{split}
\]

Let us describe the inverse of $c$ starting by the case $n =
2$. Let $\tabdimijk$ be a table of dimensions of width $2$ and let $(\xb, \yb)$ be an
element of $\Xt{i} \fibrxt{k} \Xt{j}$. By definition,
we have $\Glst[k]{i}(\xb) = \Gltt[k]{j}(\yb)$.
Since $\Glst[k]{i} = \Glst{k+1}\Gltt[k+1]{i}$, this means that
\[ \big(x_1, \dots, x_k, x_{k+1} \comp_k^{k+1} \Glt{k+2}(x_{k+2})\big) = 
\big(y_1, \dots, y_{k+1}\big), \]
i.e., that
\begin{equation}\label{eq:fibr_prod}
\tag{$\ast$}
\begin{split}
y_l & = x_l,\quad 1 \le l \le k,\\
y_{k+1} & = x_{k+1} \comp_k^{k+1} \Glt{k+2}(x_{k+2}).
\end{split}
\end{equation}
The inverse
\[ c^{-1} :  \Xt{i} \fibrx{k} \Xt{k+1,j} \to \Xt{i} \fibrxt{k} \Xt{j} \]
is thus given by the formula
\[
\begin{split}
\MoveEqLeft
c^{-1}\big(x_1, \dots, x_{i+1}, y_{k+2}, \dots, y_{j+1}\big) \\
& =
\big(x_1, \dots, x_{i+1}, x_1, \dots, x_k, x_{k+1} \comp_k^{k+1}
\Glt{k+2}(x_{k+2}), y_{k+2}, \dots, y_{j+1}\big).
\end{split}
\]

In the general case, the inverse
\[
c^{-1} : \Xt{i_1} \fibrx{i'_1} \Xt{i'_1+1,i_2} \fibrx{i'_2} \dots
\fibrx{i'_{n-1}} \Xt{i'_{n-1}+1, i_n}
\to
\Xt{i_1} \fibrxt{i'_1} \dots  \fibrxt{i'_{n-1}} \Xt{i_n}
\]
is given by the formula
\[
\begin{split}
\MoveEqLeft
c^{-1}\big(x^1_1, \dots, x^1_{i_1 + 1}, x^2_{i'_1+2}, \dots, x^2_{i_2 + 1},
\dots, x^n_{i'_{n-1} + 2}, \dots, x^n_{i_n + 1}\big) \\
& = \big(x^1_1, \dots, x^1_{i_1 + 1}, x^2_1, \dots, x^2_{i_2 + 1},
\dots, x^n_1, \dots, x^n_{i_n + 1}\big),
\end{split}
\]
where the
\[ x^l_j, \quad 2 \le l \le n, \quad 1 \le j \le i'_l+1, \]
are defined (by induction on $l$) by
\[
\begin{split}
x_j^{l+1} & = x^l_j, \quad 1 \le j \le i'_l, \\
x^{l+1}_{i'_l+1} & = x^l_{i'_l+1} \comp^{i'_l+1}_{i'_l}
\Glt{i'_l+2}(x^l_{i'_l+2}).
\end{split}
\]
\end{paragr}

\begin{paragr}\label{paragr:shifted_under_thz}
Since $(C, K)$ is a globular extension, by the universal property of $\Thz$
(Proposition \ref{prop:prop_univ_thz}), we can lift $K$ to a globular
functor $K_0 : \Thz \to C$ defined up to a unique isomorphism. 
Suppose now that a globular lifting $F_0 : \Thz \to C$ to $F$ is given.
Proposition \ref{prop:shifted_ge} allows us to express globular sums of $(C,
K)$ in terms of those of $(C, F)$. The globular lifting $K_0 : \Thz \to C$
is hence uniquely determined by $F_0$. We will call $(C, K_0)$ the
\ndef{twisted globular extension under $\Thz$} of $(C, F_0)$.
\end{paragr}

\section{Shifted groupoidal globular extensions}\label{sec:shifted_gre}

\begin{paragr}
In this section, we fix a pregroupoidal globular extension $(C, F_0)$.
In particular, the globular extension $C$ is endowed with morphisms
\[ \Thn{i} = \Thn[i-1]{i}, \quad i \ge 1, \]
and we can thus apply the previous section
and in particular Proposition~\ref{prop:shifted_ge} and Paragraph
\ref{paragr:shifted_under_thz} to get a twisted globular extension $(C,
K_0)$ under $\Thz$.

The purpose of the section is to put (under some assumptions) a
structure of pregroupoidal globular extension on $(C, K_0)$ and to prove that
if $(C, F_0)$ is groupoidal, then so is $(C, K_0)$. In the latter case, we
will call $(C, K_0)$ (endowed with its additional structure) the
\ndef{twisted groupoidal globular extension} of $(C, K_0)$.
\end{paragr}

\begin{paragr}\label{paragr:def_thn}
We define morphisms
\[
\begin{split}
\Thnt[j]{i} & : \Dnt{i} \to \Dnt{i} \amalgdt{j} \Dnt{i} = \Dnt{i} \amalgd{j}
\Dnt{j+1, i}, \quad i > j \ge 0,\\
\Thkt{i} & : \Dnt{i+1} \to \Dnt{i}, \quad i \ge 0,\\
\Thwt[j]{i} & : \Dnt{i} \to \Dnt{i}, \quad i > j \ge 0,\\
\end{split}
\]
by the formulas
\[
\begin{split}
\Thnt[j]{i} & =
\left(\cepsg{1}, \dots, \cepsg{j+1}, \left(\cepsg{j+2},
\cepsp{j+2}\right)\Thn[j]{j+2}, \dots, \left(\cepsg{i+1},
\cepsp{i+1}\right)\Thn[j]{i+1}\right),\\*
& \phantom{=1} \text{where $\cepsp{k}$ denotes $\cepsg{k+i-j}$,} \\
\Thkt{i} & =
\left(\ceps{1}, \dots, \ceps{i+1}, \ceps{i+1}\Ths{i+1}\Thk{i}\Thk{i+1}\right),\\
\Thwt[j]{i} & =
\left(\ceps{1}, \dots, \ceps{j}, \left(\ceps{j+1},
\ceps{j+2}\Tht{j+2}\right)\Thn{j+1}, \ceps{j+2}\Thw[j]{j+2}, \dots,
\ceps{i+1}\Thw[j]{i+1}\right).
\end{split}
\]
Note that $\cepsp{k} : \Dn{k} \to \Dnt{i} \amalgd{j} \Dnt{j+1, i}$ is the
canonical morphism corresponding to the factor $\Dn{k}$ of $\Dnt{j+1, i}$.
In the sequel of this section, $(C, K_0)$ will denote the globular extension
$(C, K_0)$ under $\Thz$ endowed with these $\Thnt[j]{i}$'s, $\Thkt{i}$'s and
$\Thwt[j]{i}$'s.

Dually, we define maps
\[
\begin{split}
\compt_j^i & : \Xt{i} \fibrxt{j} \Xt{i} = \Xt{i} \fibrx{j} \Xt{j+1,i} \to
\Xt{i}, \quad i > j \ge 0,\\
\Glkt{i} & : \Xt{i} \to \Xt{i+1}, \quad i \ge 0,\\
\Glwt[j]{i} & : \Xt{i} \to \Xt{i}, \quad i > j \ge 0,\\
\end{split}
\]
by the formulas
\[
\xb \Glnt[j]{i} \yb =
\big(x_1, \dots, x_{j+1}, x_{j+2} \comp_j^{j+2} y_{j+2}, \dots, x_{i+1}
\comp_j^{i+1} y_{i+1}\big),
\]
where $(\xb, \yb)$ is in $\Xt{i} \fibrxt{j} \Xt{i}$,
and
\[
\begin{split}
\Glkt{i}(\xb) & = \big(x_1, \dots, x_{i+1},
\Glk{i+1}\Glk{i}\Gls{i+1}(x_{i+1})\big), \\
\Glwt[j]{i}(\xb) & =
\big(x_1, \dots, x_j, x_{j+1} \comp_j^{j+1} \Glt{j+2}(x_{j+2}),
\Glw[j]{j+2}(x_{j+2}), \dots, \Glw[j]{i+1}(x_{i+1})\big),
\end{split}
\]
where $\xb$ is in $\Xt{i}$.
\end{paragr}

\begin{prop}\label{prop:thnt}
The $\Thnt[j]{i}$'s are well-defined. Moreover, if $(C, F_0)$ satisfies
Axioms $\Assx$ and $\Excx$, then the $\Thnt[j]{i}$'s have the desired
globular source and target, i.e., they satisfy Condition
\emph{(\ref{item:sbg_n})} of the definition of a precategorical globular
extension (see Paragraph \ref{paragr:precat}).
\end{prop}

\begin{proof}
Recall that we have fixed a globular presheaf $X$ on $C$.
Let $i > j \ge 0$.
By Paragraph
\ref{paragr:element}, showing that $\Thn[j]{i}$ is well-defined is
equivalent to showing that for every $(\xb, \yb)$ in $\Xt{i} \fibrxt{j}
\Xt{i}$, the element $\xb \Glnt[j]{i} \yb$ belongs to
$\Xt{i}$.

Let us show this. Let $(\xb, \yb)$ be in $\Xt{i} \fibrxt{j} \Xt{i}$.
We need to check that
\[ \Gls{j+1}(x_{j+1}) = \Glt{j+1}\Glt{j+2}(x_{j+2} \comp_j^{j+2} y_{j+2}), \]
and
\[
\Gls{l}(x_l \comp_j^l y_l) = \Glt{l}\Glt{l+1}
(x_{l+1} \comp_j^{l+1} y_{l+1}),
\quad j + 2 \le l \le i.
\]
But
\[
\begin{split}
\Glt{j+1}\Glt{j+2}(x_{j+2} \comp_j^{j+2} y_{j+2})
& =
\Glt{j+1}\big(\Glt{j+2}(x_{j+2}) \comp_j^{j+1} \Glt{j+2}(y_{j+2})\big) \\
& =
\Glt{j+1}\Glt{j+2}(x_{j+2}) \\
& =
\Gls{j+1}(x_{j+1}),
\end{split}
\]
and
\[
\begin{split}
\Gls{l}(x_l \comp_j^l y_l)
& =
\Gls{l}(x_l) \comp_j^{l-1} \Gls{l}(y_l) \\
& =
\Glt{l}\Glt{l+1}(x_{l+1}) \comp_j^{l-1} \Glt{l}\Glt{l+1}(y_{l+1}) \\
& =
\Glt{l}\Glt{l+1}(x_{l+1} \comp_j^{l+1} y_{l+1}).
\end{split}
\]

Again by Paragraph \ref{paragr:element}, proving that $\Thn[j]{i}$ has the desired
source and target is equivalent to proving the analogous result for $\xb
\Glnt[j]{i} \yb$.

Let us prove this.  If $j = i - 1$, we have
\begin{align*}
\Glst{i}\big(\xb \Glnt[i-1]{i} \yb\big)
& =
\Glst{i}\big(x_1, \dots, x_i, x_{i+1} \comp^{i+1}_{i-1} y_{i+1}\big) \\
& =
\big(x_1, \dots, x_{i-1}, x_i \comp^i_{i-1} \Glt{i}(x_{i+1}
\comp^{i+1}_{i-1} y_{i+1})\big) \\
& =
\big(x_1, \dots, x_{i-1}, x_i \comp^i_{i-1} \big(\Glt{i}(x_{i+1})
\comp^i_{i-1}\Glt{i}(y_{i+1})\big)\big)\\
& =
\big(x_1, \dots, x_{i-1}, \big(x_i \comp^i_{i-1} \Glt{i}(x_{i+1})\big)
\comp^i_{i-1}\Glt{i}(y_{i+1})\big)\\*
& \phantom{=1} \text{(by Axiom $\Ass{i}{i-1}$)} \\
& =
\big(y_1, \dots, y_{i-1}, y_i \comp^i_{i-1}\Glt{i}(y_{i+1})\big) \\*
& \phantom{=1}\text{(by Equations (\ref{eq:fibr_prod}) of Paragraph
\ref{paragr:iso_can})} \\*
& =
\Glst{i}(\yb),
\end{align*}
and
\[
\begin{split}
\Gltt{i}\big(\xb \Glnt[i-1]{i} \yb\big)
& =
\Gltt{i}\big(x_1, \dots, x_i, x_{i+1} \comp^{i+1}_{i-1} y_{i+1}\big) \\
& =
\big(x_1, \dots, x_i\big) \\
& =
\Gltt{i}(\xb).
\end{split}
\]
If $j < i - 1$, we have
\begin{align*}
\Glst{i}\big(\xb \Glnt[j]{i} \yb\big)
& =
\Glst{i}\big(x_1, \dots, x_{j+1}, x_{j+2} \comp^{j+2}_j y_{j+2}, \dots,
x_{i+1} \comp_j^{i+1} y_{i+1}\big) \\
& =
\big(x_1, \dots, x_{j+1}, x_{j+2} \comp^{j+2}_j y_{j+2}, \dots,
x_{i-1} \comp_j^{i-1} y_{i-1}, \\*
& \qquad\quad
(x_i \comp^i_j y_i) \comp^i_{i-1}
\Glt{i+1}(x_{i+1} \comp^{i+1}_j y_{i+1})\big) \\
& =
\big(x_1, \dots, x_{j+1}, x_{j+2} \comp^{j+2}_j y_{j+2}, \dots,
x_{i-1} \comp_j^{i-1} y_{i-1}, \\*
& \qquad\quad
\big(x_i \comp^i_j y_i\big) \comp^i_{i-1}
\big(\Glt{i+1}(x_{i+1}) \comp^i_j \Glt{i+1}(y_{i+1})\big)\big) \\
& =
\big(x_1, \dots, x_{j+1}, x_{j+2} \comp^{j+2}_j y_{j+2}, \dots,
x_{i-1} \comp_j^{i-1} y_{i-1}, \\*
& \qquad\quad
\big(x_i \comp^i_{i-1} \Glt{i+1}(x_{i+1})\big) \comp^i_j
\big(y_i \comp^i_{i-1} \Glt{i+1}(y_{i+1})\big)\big) \\*
& \phantom{=1}\text{(by Axiom $\Exc{i}{i-1}{j}$)} \\
& =
\big(x_1, \dots, x_{i-1}, x_i \comp^i_{i-1} \Glt{i+1}(x_{i+1})\big)
\Glnt[j]{i} \\*
& \qquad\quad
\big(y_1, \dots, y_{i-1}, y_i \comp^i_{i-1} \Glt{i+1}(y_{i+1})\big) \\*
& =
\Glst{i}(\xb) \Glnt[j]{i} \Glst{i}(\yb),
\end{align*}
and
\[
\begin{split}
\Gltt{i}\big(\xb \Glnt[j]{i} \yb\big)
& =
\Gltt{i}\big(x_1, \dots, x_{j+1}, x_{j+2} \comp^{j+2}_j y_{j+2}, \dots,
x_{i+1} \comp_j^{i+1} y_{i+1}\big) \\
& =
\big(x_1, \dots, x_{j+1}, x_{j+2} \comp^{j+2}_j y_{j+2}, \dots,
x_i \comp_j^i y_i\big) \\
& =
(x_1, \dots, x_i) \Glnt[j]{i-1} (y_1, \dots, y_i) \\
& =
\Gltt{i}(\xb) \Glnt[j]{i-1} \Gltt{i}(\yb),
\end{split}
\]
hence the result.
\end{proof}

\begin{prop}\label{prop:ass}
If $(C, F_0)$ satisfies Axiom $\Assx$, then so does $(C, K_0)$.
\end{prop}

\begin{proof}
Let $i > j \ge 0$ and let $(\xb, \yb, \zb)$ be in 
$\Xt{i} \fibrxt{j} \Xt{i} \fibrxt{j} \Xt{i}$. We have
\begin{align*}
\big(\xb \Glnt[j]{i} \yb\big) \Glnt[j]{i} \zb
& =
\big(x_1, \dots, x_{j+1}, x_{j+2} \comp_j^{j+2} y_{j+2}, \dots, x_{i+1}
\comp_j^{i+1} y_{i+1}\big) \Glnt[j]{i} \zb \\
& =
\big(x_1, \dots, x_{j+1}, \\*
& \qquad\quad
(x_{j+2} \comp_j^{j+2} y_{j+2}) \comp_j^{j+2}
z_{j+2}, \dots, (x_{i+1} \comp_j^{i+1} y_{i+1}) \comp_j^{i+1} z_{i+1}\big)
\\
& =
\big(x_1, \dots, x_{j+1}, \\*
& \qquad\quad
x_{j+2} \comp_j^{j+2} (y_{j+2} \comp_j^{j+2}
z_{j+2}), \dots, x_{i+1} \comp_j^{i+1} (y_{i+1} \comp_j^{i+1} z_{i+1})\big)
\\*
& \phantom{=1} \text{(by Axiom $\Ass{l}{j}$ for $j + 2 \le l \le i + 1$)} \\
& =
\xb \Glnt[j]{i} \big(y_1, \dots, y_{j+1}, y_{j+2} \comp_j^{j+2} z_{j+2}, \dots, y_{i+1}
\comp_j^{i+1} z_{i+1}\big) \\*
& =
\xb \Glnt[j]{i} \big(\yb \Glnt[j]{i} \zb\big).
\end{align*}
\end{proof}

\begin{prop}\label{prop:exc}
If $(C, F_0)$ satisfies Axiom $\Excx$, then so does $(C, K_0)$.
\end{prop}

\begin{proof}
Let $i > j > k \ge 0$ and let $(\xb, \yb, \zb, \tb)$ be in
\[ \Xt{i} \fibrxt{j} \Xt{i} \fibrxt{j} \Xt{i} \fibrxt{j} \Xt{i}.\]
We have
\begin{align*}
\MoveEqLeft \big(\xb \Glnt[j]{i} \yb) \Glnt[k]{i} \big(\zb \Glnt[j]{i}
\tb\big) \\*
& =
\big(x_1, \dots, x_{j+1}, x_{j+2} \comp_j^{j+2} y_{j+2}, \dots, x_{i+1}
\comp_j^{i+1} y_{i+1}\big) \Glnt[k]{i} \\*
& \qquad\quad
\big(z_1, \dots, z_{j+1}, z_{j+2} \comp_j^{j+2} t_{j+2}, \dots, z_{i+1}
\comp_j^{i+1} t_{i+1}\big) \\
& =
(x_1, \dots, x_{k+1}, x_{k+2} \comp_k^{k+2} z_{k+2}, \dots, x_{j+1}
\comp_k^{j+1} z_{j+1}, \\*
& \qquad\quad
\big(x_{j+2} \comp_j^{j+2} y_{j+2}\big) \comp_k^{j+2} \big(z_{j+2}
\comp_j^{j+2} t_{j+2}\big), \dots, \\*
& \qquad\quad
\big(x_{i+1} \comp_j^{i+1} y_{i+1}\big) \comp_k^{i+1} \big(z_{i+1}
\comp_j^{i+1} t_{i+1}\big)\big) \\
& =
(x_1, \dots, x_{k+1}, x_{k+2} \comp_k^{k+2} z_{l+2}, \dots, x_{j+1}
\comp_k^{j+1} z_{j+1}, \\*
& \qquad\quad
\big(x_{j+2} \comp_k^{j+2} z_{j+2}\big) \comp_j^{j+2} \big(y_{j+2}
\comp_k^{j+2} t_{j+2}\big), \dots, \\*
& \qquad\quad
\big(x_{i+1} \comp_k^{i+1} z_{i+1}\big) \comp_j^{i+1} \big(y_{i+1}
\comp_k^{i+1} t_{i+1}\big)\big) \\*
& \phantom{=1}\text{(by Axiom $\Exc{l}{j}{k}$ for $l$ such that $j+2 \le l \le i+1$)}\\
& =
\big(x_1, \dots, x_{k+1}, x_{k+2} \comp_k^{k+2} z_{k+2}, \dots, x_{i+1}
\comp_k^{i+1} z_{i+1}\big) \Glnt[j]{i} \\*
& \qquad\quad
\big(y_1, \dots, y_{k+1}, y_{k+2} \comp_k^{k+2} t_{k+2}, \dots, y_{i+1}
\comp_k^{i+1} t_{i+1}\big) \\*
& =
\big(\xb \Glnt[k]{i} \zb\big) \Glnt[j]{i} \big(\yb \Glnt[k]{i} \tb\big).
\end{align*}
\end{proof}

\begin{prop}\label{prop:thkt}
The $\Thkt{i}$'s are well-defined. Moreover, if $(C, F_0)$ satisfies
Axiom $\Runx$, then the $\Thkt{i}$'s have the desired globular source and
target, i.e., they
satisfy Condition \emph{(\ref{item:sbg_k})} of the definition of a precategorical
globular extension (see Paragraph \ref{paragr:precat}).
\end{prop}

\begin{proof}
Let $i \ge 0$ and let $\xb$ be in $\Xt{i}$. Let us first prove that $\Glkt{i}(\xb)$
belongs to $\Xt{i+1}$. We need to show that
\[ \Gls{i+1}(x_{i+1}) =
\Glt{i+1}\Glt{i+2}\Glk{i+1}\Glk{i}\Gls{i+1}(x_{i+1}),
\]
but this identity holds since $\Glt{l+1}\Glk{l} = \id{X_l}$ for every $l \ge
0$.

Let us now prove that $\Glkt{i}(\xb)$ has the desired globular source and target. We
have
\[
\begin{split}
\Glst{i+1}\Glkt{i}(\xb) 
& =
\Glst{i+1}\big(x_1, \dots, x_{i+1}, \Glk{i+1}\Glk{i}\Gls{i+1}(x_{i+1})\big) \\
& =
\big(x_1, \dots, x_i, x_{i+1} \comp^{i+1}_i
\Glt{i+2}\Glk{i+1}\Glk{i}\Gls{i+1}(x_{i+1})\big) \\
& =
\big(x_1, \dots, x_i, x_{i+1} \comp^{i+1}_i
\Glk{i}\Gls{i+1}(x_{i+1})\big) \\
& =
\big(x_1, \dots, x_{i+1}\big) \\*
& \phantom{=1} \text{(by Axiom $\Run{i+1}{i}$)} \\
& =
\xb,
\end{split}
\]
and
\[
\begin{split}
\Gltt{i+1}\Glkt{i}(\xb) 
& =
\Gltt{i+1}\big(x_1, \dots, x_{i+1}, \Glk{i+1}\Glk{i}\Gls{i+1}(x_{i+1})\big) \\
& =
\big(x_1, \dots, x_{i+1}\big) \\
& =
\xb,
\end{split}
\]
hence the result.
\end{proof}

\begin{prop}\label{prop:lrun}
If $(C, F_0)$ satisfies Axioms $\Lunx$ and $\Runx$, then so does $(C, K_0)$.
\end{prop}

\begin{proof}
Let $j \ge 0$ and let $\yb$ be in $\Xt{j}$. Let us first prove by induction on $i >
j$ 
that
\[
\Glkt[i]{j}(\yb) = \big(y_1, \dots, y_{j+1},
\Glk[j+2]{j}\Gls{j+1}(y_{j+1}), \dots, \Glk[i+1]{j}\Gls{j+1}(y_{j+1})\big).
\]
For $i = j + 1$, this identity holds by definition of $\Glkt{j}$. Assume 
the result holds for an $i > j$. Then we have
\[
\begin{split}
\Glkt[i+1]{j}(\yb) 
& =
\Glkt{i}\Glkt[i]{j}(\yb) \\
& =
\Glkt{i}
\big(y_1, \dots, y_{j+1},
\Glk[j+2]{j}\Gls{j+1}(y_{j+1}), \dots, \Glk[i+1]{j}\Gls{j+1}(y_{j+1})\big) \\
& =
\big(y_1, \dots, y_{j+1},
\Glk[j+2]{j}\Gls{j+1}(y_{j+1}), \dots, \Glk[i+1]{j}\Gls{j+1}(y_{j+1}), \\
& \qquad\quad
\Glk{i+1}\Glk{i}\Gls{i+1}\Glk[i+1]{j}\Gls{j+1}(y_{j+1})
\big),
\end{split}
\]
but
\[
\begin{split}
\Glk{i+1}\Glk{i}\Gls{i+1}\Glk[i+1]{j}\Gls{j+1}(y_{j+1})
& =
\Glk{i+1}\Glk{i}\Gls{i+1}\Glk{i}\Glk[i]{j}\Gls{j+1}(y_{j+1}) \\
& =
\Glk{i+1}\Glk{i}\Glk[i]{j}\Gls{j+1}(y_{j+1}), \\
& =
\Glk[i+2]{j}\Gls{j+1}(y_{j+1}),
\end{split}
\]
hence the formula.

Let now $i > j$ and let $\xb$ be in $\Xt{i}$. We have 
\[
\begin{split}
\Glkt[i]{j}\Glst[j]{i}(\xb)
& = \Glkt[i]{j}\Glst{j+1}\Gltt[j+1]{i}(\xb) \\
& =
\Glkt[i]{j}\Glst{j+1}\big(x_1, \dots, x_{j+2}\big) \\
& =
\Glkt[i]{j}\big(x_1, \dots, x_j, x_{j+1} \comp_j^{j+1} \Glt{j+2}(x_{j+2})\big) \\
& =
\big(x_1, \dots, x_j, x_{j+1} \comp_j^{j+1} \Glt{j+2}(x_{j+2}), \\
& \qquad\quad
\Glk[j+2]{j}\Gls{j+1}\big(x_{j+1} \comp_j^{j+1} \Glt{j+2}(x_{j+2})\big),
\dots, \\
& \qquad\quad
\Glk[i+1]{j}\Gls{j+1}\big(x_{j+1} \comp_j^{j+1}
\Glt{j+2}(x_{j+2})\big)\big).
\end{split}
\]
But for $l$ such that $j+2 \le l \le i + 1$, we have
\[
\begin{split}
\Glk[l]{j}\Gls{j+1}\big(x_{j+1} \comp_j^{j+1}
\Glt{j+2}(x_{j+2})\big) 
& =
\Glk[l]{j}\Gls{j+1}\Glt{j+2}(x_{j+2}) \\
& =
\Glk[l]{j}\Gls[j]{j+2}(x_{j+2}) \\
& =
\Glk[l]{j}\Gls[j]{l}(x_l),
\end{split}
\]
where the last equality comes from Lemma \ref{lemma:struct_xt}. Hence the identity
\begin{equation}\label{eq:ks}
\tag{$\ast_{ks}$}
\begin{split}
\Glkt[i]{j}\Glst[j]{i}(\xb)
& =
\big(x_1, \dots, x_j, x_{j+1} \comp_j^{j+1} \Glt{j+2}(x_{j+2}), \\
& \qquad\quad
\Glk[j+2]{j}\Gls[j]{j+2}(x_{j+2}), \dots,
\Glk[i+1]{j}\Gls[j]{i+1}(x_{i+1})\big).
\end{split}
\end{equation}

Let us now compute $\Glkt[i]{j}\Gltt[j]{i}(\xb)$. We have
\[
\begin{split}
\Glkt[i]{j}\Gltt[j]{i}(\xb)
& =
\Glkt[i]{j}\big(x_1, \dots, x_{j+1}\big) \\
& =
\big(x_1, \dots, x_{j+1}, \Glk[j+2]{j}\Gls{j+1}(x_{j+1}), \dots,
\Glk[i+1]{j}\Gls{j+1}(x_{j+1})\big).
\end{split}
\]
But by Lemma \ref{lemma:struct_xt}, we have
\[ \Gls{j+1}(x_{j+1}) = \Glt[j]{l}(x_l), \quad j+2 \le l \le i+1, \]
and so we obtain the formula
\begin{equation}\label{eq:kt}
\tag{$\ast_{kt}$}
\Glkt[i]{j}\Gltt[j]{i}(\xb) =
\big(x_1, \dots, x_{j+1}, \Glk[j+2]{j}\Glt[j]{j+2}(x_{j+2}), \dots,
\Glk[i+1]{j}\Glt[j]{i+1}(x_{i+1})\big).
\end{equation}

We can now prove the proposition. We have
\[
\begin{split}
\xb \Glnt[j]{i} \Glkt[i]{j}\Glst[j]{i}(\xb)
& =
\xb \Glnt[j]{i}
\big(x_1, \dots, x_j, x_{j+1} \comp_j^{j+1} \Glt{j+2}(x_{j+2}), \\
& \qquad\qquad\quad
\Glk[j+2]{j}\Gls[j]{j+2}(x_{j+2}), \dots,
\Glk[i+1]{j}\Gls[j]{i+1}(x_{i+1})\big) \\
& =
\big(x_1, \dots, x_{j+1}, \\
& \qquad\quad
x_{j+2} \comp_j^{j+2} \Glk[j+2]{j}\Gls[j]{j+2}(x_{j+2}), \dots,
x_{i+1} \comp_j^{i+1} \Glk[i+1]{j}\Gls[j]{i+1}(x_{i+1})\big) \\
& =
\big(x_1, \dots, x_{i+1}\big) \\*
& \phantom{=1}\text{(by Axioms $\Run{l}{j}$ for $l$ such that $j+2 \le l \le i+1$)} \\
& =
\xb,
\end{split}
\]
and
\[
\begin{split}
\Glkt[i]{j}\Gltt[j]{i}(\xb) \Glnt[j]{i} \xb
& =
\big(x_1, \dots, x_{j+1}, \Glk[j+2]{j}\Glt[j]{j+2}(x_{j+2}), \dots,
\Glk[i+1]{j}\Glt[j]{i+1}(x_{i+1})\big)
\Glnt[j]{i} \xb \\
& =
\big(x_1, \dots, x_{j+1},  \\
& \qquad\quad
\Glk[j+2]{j}\Glt[j]{j+2}(x_{j+2}) \comp_j^{j+2} x_{j+2}, \dots,
\Glk[i+1]{j}\Glt[j]{i+1}(x_{i+1}) \comp_j^{i+1} x_{i+1}\big) \\
& =
\big(x_1, \dots, x_{i+1}\big) \\*
& \phantom{=1}\text{(by Axioms $\Lun{l}{j}$ for $l$ such that $j+2 \le l \le
i+1$)} \\
& =
\xb.
\end{split}
\]
\end{proof}

\begin{prop}\label{prop:fun}
If $(C, F_0)$ satisfies Axiom $\Funx$, then so does $(C, K_0)$.
\end{prop}

\begin{proof}
Let $i > j \ge 0$ and let $(\xb, \yb)$ be in $\Xt{i} \fibrxt{j} \Xt{i}$. We have
\[
\begin{split}
\Glkt{i}\big(\xb \Glnt[j]{i} \yb\big) 
& =
\Glkt{i}\big(x_1, \dots, x_{j+1}, x_{j+2} \comp_j^{j+2} y_{j+2}, \dots,
x_{i+1} \comp_j^{i+1} y_{i+1}\big) \\
& =
\big(x_1, \dots, x_{j+1}, x_{j+2} \comp_j^{j+2} y_{j+2}, \dots,
x_{i+1} \comp_j^{i+1} y_{i+1}, \\
& \qquad\quad \Glk[i+2]{i}\Gls{i+1}(x_{i+1} \comp_j^{i+1}
y_{i+1})\big) \\
& =
\big(x_1, \dots, x_{j+1}, x_{j+2} \comp_j^{j+2} y_{j+2}, \dots,
x_{i+1} \comp_j^{i+1} y_{i+1}, \\
& \qquad\quad \Glk[i+2]{i}(\Gls{i+1}(x_{i+1}) \comp_j^i
\Gls{i+1}(y_{i+1}))\big) \\
& =
\big(x_1, \dots, x_{j+1}, x_{j+2} \comp_j^{j+2} y_{j+2}, \dots,
x_{i+1} \comp_j^{i+1} y_{i+1}, \\
& \qquad\quad (\Glk[i+2]{i}\Gls{i+1}(x_{i+1}) \comp_j^{i+2}
\Glk[i+2]{i}\Gls{i+1}(y_{i+1}))\big) \\*
& \phantom{=1} \text{(by Axioms $\Fun{i}{j}$ and $\Fun{i+1}{j}$)} \\
& =
\big(x_1, \dots, x_{i+1}, \Glk[i+2]{i}\Gls{i+1}(x_{i+1})\big)
\Glnt[j]{i}
\big(y_1, \dots, y_{i+1}, \Glk[i+2]{i}\Gls{i+1}(y_{i+1})\big) \\
& =
\Glkt{i}\big(\xb\big) \Glnt[j]{i+1} \Glkt{i}\big(\yb\big).
\end{split}
\]
\end{proof}

\begin{prop}
The $\Thwt[j]{i}$'s are well-defined. Moreover, if $(C, F_0)$ satisfies
Axioms $\Assx$, $\Excx$, $\Lunx$, $\Runx$ and $\RInvx$, then the
$\Thwt[j]{i}$'s have the desired globular source and target, i.e., they
satisfy the condition of the definition of a pregroupoidal globular
extension (see Paragraph \ref{paragr:pregr}).
\end{prop}

\begin{proof}
Note that by the remark at the end of Paragraph \ref{paragr:def_oo-grpd},
$(C, F_0)$ also satisfies Axiom $\FInvx$.

Let $i > j \ge 0$ and let $\xb$ be in $\Xt{i}$. Let us first prove that
$\Glwt[j]{i}(\xb)$ belongs to $\Xt{i}$. We need to show that
\[
\begin{split}
\Gls{j}(x_j) & = \Glt{j}\Glt{j+1}\big(x_{j+1} \comp_j^{j+1}
\Glt{j+2}(x_{j+2})\big),\\
\Gls{j+1}\big(x_{j+1} \comp_j^{j+1} \Glt{j+2}(x_{j+2})\big) &
= \Glt{j+1}\Glt{j+2}\Glw[j]{j+2}(x_{j+2}),
\end{split}
\]
and
\[
\Gls{l}\Glw[j]{l}(x_l) = \Glt{l}\Glt{l+1}\Glw[j]{l+1}(x_{l+1}),
\quad j+2 \le l \le i.
\]
The first identify has already been proved in Paragraph
\ref{paragr:element}. The two others follow from the following calculations:
\[
\begin{split}
\Gls{j+1}\big(x_{j+1} \comp_j^{j+1} \Glt{j+2}(x_{j+2})\big)
& =
\Gls{j+1}\Glt{j+2}(x_{j+2}) \\
& =
\Glt{j+1}\Glw[j]{j+1}\Glt{j+2}(x_{j+2}) \\
& =
\Glt{j+1}\Glt{j+2}\Glw[j]{j+2}(x_{j+2}),
\end{split}
\]
and
\[
\begin{split}
\Gls{l}\Glw[j]{l}(x_l)
& =
\Glw[j]{l-1}\Gls{l}(x_l) \\
& =
\Glw[j]{l-1}\Glt{l}\Glt{l+1}(x_{l+1}) \\
& =
\Glt{l}\Glt{l+1}\Glw[j]{l+1}(x_{l+1}).
\end{split}
\]

Let us now prove that $\Glwt[j]{i}(\xb)$ has the desired globular source and target.
For $j = i - 1$, we have
\begin{align*}
\Glst{i}\Glwt[i-1]{i}(\xb)
& =
\Glst{i}\big(x_1, \dots, x_{i-1}, x_i \comp^i_{i-1} \Glt{i+1}(x_{i+1}),
\Glw[i-1]{i+1}(x_{i+1})\big) \\
& =
\big(x_1, \dots, x_{i-1}, \big(x_i \comp^i_{i-1} \Glt{i+1}(x_{i+1})\big)
\comp^i_{i-1} \Glt{i+1}\Glw[i-1]{i+1}(x_{i+1})\big) \\
& =
\big(x_1, \dots, x_{i-1}, \big(x_i \comp^i_{i-1} \Glt{i+1}(x_{i+1})\big)
\comp^i_{i-1} \Glw[i-1]{i}\Glt{i+1}(x_{i+1})\big) \\
& =
\big(x_1, \dots, x_{i-1}, x_i \comp^i_{i-1} \big( \Glt{i+1}(x_{i+1})
\comp^i_{i-1} \Glw[i-1]{i}\Glt{i+1}(x_{i+1})\big)\big) \\*
& \phantom{=1} \text{(by Axiom $\Ass{i}{i-1}$)} \\
& =
\big(x_1, \dots, x_{i-1}, x_i \comp^i_{i-1}
\Glk{i-1}\Glt{i}\Glt{i+1}(x_{i+1})\big) \\*
& \phantom{=1} \text{(by Axiom $\RInv{i}{i-1}$)} \\
& =
\big(x_1, \dots, x_{i-1}, x_i \comp^i_{i-1}
\Glk{i-1}\Gls{i}(x_i)\big) \\
& =
\big(x_1, \dots, x_i\big) \\*
& \phantom{=1} \text{(by Axiom $\Run{i}{i-1}$)} \\*
& =
\Gltt{i}(\xb),
\end{align*}
and
\[
\begin{split}
\Gltt{i}\Glwt[i-1]{i}(\xb) & =
\Gltt{i}\big(x_1, \dots, x_{i-1}, x_i \comp^i_{i-1} \Glt{i+1}(x_{i+1}),
\Glw[i-1]{i+1}(x_{i+1})\big) \\
& =
\big(x_1, \dots, x_{i-1}, x_i \comp^i_{i-1} \Glt{i+1}(x_{i+1})\big) \\
& =
\Glst{i}(\xb).
\end{split}
\]
For $j < i - 1$, we have
\begin{align*}
\Glst{i}\Glwt[j]{i}(\xb)
& =
\Glst{i}\big(x_1, \dots, x_j, x_{j+1} \comp^{j+1}_j \Glt{j+2}(x_{j+2}),
\Glw[j]{j+2}(x_{j+2}), \dots, \Glw[j]{i+1}(x_{i+1})\big) \\
& =
\big(x_1, \dots, x_j, x_{j+1} \comp^{j+1}_j \Glt{j+2}(x_{j+2}), \\
& \qquad\quad 
\Glw[j]{j+2}(x_{j+2}), \dots, \Glw[j]{i-1}(x_{i-1}),
\Glw[j]{i}(x_i) \comp^i_{i-1} \Glt{i+1}\Glw[j]{i+1}(x_{i+1})\big) \\
& =
\big(x_1, \dots, x_j, x_{j+1} \comp^{j+1}_j \Glt{j+2}(x_{j+2}), \\
& \qquad\quad 
\Glw[j]{j+2}(x_{j+2}), \dots, \Glw[j]{i-1}(x_{i-1}),
\Glw[j]{i}(x_i) \comp^i_{i-1} \Glw[j]{i}\Glt{i+1}(x_{i+1})\big) \\
& =
\big(x_1, \dots, x_j, x_{j+1} \comp^{j+1}_j \Glt{j+2}(x_{j+2}), \\
& \qquad\quad 
\Glw[j]{j+2}(x_{j+2}), \dots, \Glw[j]{i-1}(x_{i-1}),
\Glw[j]{i}\big(x_i \comp^i_{i-1} \Glt{i+1}(x_{i+1})\big)\big) \\*
& \phantom{=1} \text{(by Axiom $\FInv{i}{i-1}{j}$)} \\
& =
\Glwt[j]{i-1}\big(x_1, \dots, x_{i-1}, x_i \comp^i_{i-1}
\Glt{i+1}(x_{i+1})\big) \\*
& =
\Glwt[j]{i-1}\Glst{i}(\xb),
\end{align*}
and
\[
\begin{split}
\Gltt{i}\Glwt[j]{i}(\xb)
& =
\Gltt{i}\big(x_1, \dots, x_j, x_{j+1} \comp^{j+1}_j \Glt{j+2}(x_{j+2}),
\Glw[j]{j+2}(x_{j+2}), \dots, \Glw[j]{i+1}(x_{i+1})\big) \\
& =
\big(x_1, \dots, x_j, x_{j+1} \comp^{j+1}_j \Glt{j+2}(x_{j+2}),
\Glw[j]{j+2}(x_{j+2}), \dots, \Glw[j]{i}(x_i)\big) \\
& =
\Glwt[j]{i-1}\big(x_1, \dots, x_i\big) \\
& =
\Glwt[j]{i-1}\Gltt{i}(\xb),
\end{split}
\]
hence the result.
\end{proof}

\begin{prop}
If $(C, F_0)$ satisfies Axioms $\LInvx$ and $\RInvx$, then so does $(C, K_0)$.
\end{prop}

\begin{proof}
Let $i > j \ge 0$ and let $\xb$ be in $\Xt{i}$. We have
\begin{align*}
\Glwt[j]{i}(\xb) \Glnt[j]{i} \xb
& =
\big(x_1, \dots, x_j, x_{j+1} \comp_j^{j+1} \Glt{j+2}(x_{j+2}),\\*
& \qquad\quad
\Glw[j]{j+2}(x_{j+2}), \dots, \Glw[j]{i+1}(x_{i+1})\big)
\Glnt[j]{i} \xb \\
& =
\big(x_1, \dots, x_j, x_{j+1} \comp_j^{j+1} \Glt{j+2}(x_{j+2}),\\*
& \qquad\quad
\Glw[j]{j+2}(x_{j+2}) \comp_j^{j+2} x_{j+2}, \dots, \Glw[j]{i+1}(x_{i+1})
\comp_j^{i+1} x_{i+1}\big) \\
& =
\big(x_1, \dots, x_j, x_{j+1} \comp_j^{j+1} \Glt{j+2}(x_{j+2}),\\*
& \qquad\quad
\Glk[j+2]{j}\Gls[j]{j+2}(x_{j+2}), \dots,
\Glk[i+1]{j}\Gls[j]{i+1}(x_{i+1})\big) \\*
& \phantom{=1}\text{(by Axioms $\LInv{l}{j}$ for $l$ such that $j+2 \le l
\le i+1$)} \\*
& = \Glkt[i]{j}\Glst[j]{i}(\xb),
\end{align*}
where the last equality is Equation (\ref{eq:ks}) (see the proof of
Proposition \ref{prop:lrun}), and
\[
\begin{split}
\xb \Glnt[j]{i} \Glwt[j]{i}(\xb) 
& =
\xb \Glnt[j]{i} \big(x_1, \dots, x_j, x_{j+1} \comp_j^{j+1}
\Glt{j+2}(x_{j+2}),\\
& \qquad\qquad\qquad\quad
\Glw[j]{j+2}(x_{j+2}), \dots, \Glw[j]{i+1}(x_{i+1})\big) \\
& =
\big(x_1, \dots, x_{j+1}, \\
& \qquad\quad
x_{j+2} \comp_j^{j+2} \Glw[j]{j+2}(x_{j+2}), \dots,
x_{i+1} \comp_j^{i+1} \Glw[j]{i+1}(x_{i+1})\big) \\
& =
\big(x_1, \dots, x_{j+1},
\Glk[j+2]{j}\Glt[j]{j+2}(x_{j+2}), \dots,
\Glk[i+1]{j}\Glt[j]{i+1}(x_{i+1})\big) \\*
& \phantom{=1}\text{(by Axioms $\RInv{l}{j}$ for $l$ such that $j+2 \le l
\le i+1$)} \\
& =
\Glkt[i]{j}\Gltt[j]{i}(\xb),
\end{split}
\]
where the last equality is Equation (\ref{eq:kt}) (see the proof of
Proposition \ref{prop:lrun}).
\end{proof}

\begin{coro}\label{coro:ck_gr}
If $(C, F_0)$ is groupoidal, then $(C,
K_0)$ (endowed with the $\Thnt[j]{i}$'s, $\Thkt{i}$'s and $\Thwt[j]{i}$'s) 
is a groupoidal globular extension.
\end{coro}

As announced at the beginning of this section, if $(C, F_0)$ is a groupoidal
globular extension, we will call $(C, K_0)$ the \ndef{twisted groupoidal globular
extension} of $(C, F_0)$.

\section{The décalage on $\Thtld$}\label{sec:shift_thtld}

\begin{paragr}
We now introduce the morphisms that will give rise to our décalage on~$\Thtld$.

Let $(C, F)$ be a globular extension endowed with $\Thn{i}$'s as in Section
\ref{sec:shifted_ge} and let $(C, K)$ be the twisted globular extension of
$(C, F)$.
We define morphisms
\[
\begin{split}
\alpha^{}_i & : \Dn{i} \to \Dnt{i}, \quad i \ge 0,\\
\beta^{}_i & : \Dn{0} \to \Dnt{i},\quad i \ge 0,\\
\end{split}
\]
by the formulas
\[
\begin{split}
\alpha^{}_i & =
\ceps{i+1}\Ths{i+1},\\
\beta^{}_i & =
\ceps{1}\Tht{1}.\\
\end{split}
\]
Dually, we define maps
\[
\begin{split}
a_i & : \Xt{i} \to X_i, \quad i \ge 0,\\
b_i & : \Xt{i} \to X_0, \quad i \ge 0,\\
\end{split}
\]
by the formulas
\[
\begin{split}
a_i(x_1, \dots, x_{i+1}) & = \Gls{i+1}(x_{i+1}), \\
b_i(x_1, \dots, x_{i+1}) & = \Glt{1}(x_{1}).
\end{split}
\]
\end{paragr}

\begin{prop}
The maps 
\[
\begin{split}
\Dn{i} & \mapsto \alpha^{}_i,
\quad i \ge 0,\\
\Dn{i} & \mapsto \beta^{}_i,
\quad i \ge 0,\\
\end{split}
\]
define natural transformations
\[
\xymatrix{
F \ar[r]^\alpha & K & \Dn{0} \ar[l]_{\beta}
}
\]
(where $\Dn{0}$ denotes the constant functor $\G \to C$ of value $\Dn{0}$).
\end{prop}

\begin{proof}
Let us first prove that $\alpha$ is a natural transformation. We
must show that
\[
\Thst{i}\alpha^{}_{i-1} = \alpha^{}_i\Ths{i}
\quad\text{and}\quad
\Thtt{i}\alpha^{}_{i-1} = \alpha^{}_i\Tht{i},
\quad i \ge 1.
\]
Let $i \ge 1$ and let $\xb$ be in $\Xt{i}$.
We have
\[
\begin{split}
a_{i-1}\Glst{i}(\xb) 
& =
a_{i-1}(x_1, \dots, x_{i-1}, x_i \comp^i_{i-1} \Glt{i+1}(x_{i+1}))
\\
& =
\Gls{i}(x_i \comp^i_{i-1} \Glt{i+1}(x_{i+1}))
\\
& =
\Gls{i}\Glt{i+1}(x_{i+1})
\\
& =
\Gls{i}\Gls{i+1}(x_{i+1})
\\
& =
\Gls{i}a_i(\xb),
\end{split}
\]
and
\[
\begin{split}
a_{i-1}\Gltt{i}(\xb)
& = 
a_i(x_1, \dots, x_i)
\\
& = 
\Gls{i}(x_i)
\\
& = 
\Glt{i}\Glt{i+1}(x_{i+1})
\\
& = 
\Glt{i}\Gls{i+1}(x_{i+1})
\\
& = 
\Glt{i}a_i(\xb),
\end{split}
\]
hence the naturality of $\alpha$.

To prove the naturality of $\beta$, we must check 
that
\[
\Thst{i}\beta^{}_{i-1} = \beta^{}_i
\quad\text{and}\quad
\Thtt{i}\beta^{}_{i-1} = \beta^{}_i,
\quad i \ge 1.
\]
This follows from the following calculations:
\[
\begin{split}
b_{i-1}\Glst{i}(\xb) 
& =
b_{i-1}(x_1, \dots, x_{i-1}, x_i \comp^i_{i-1} \Glt{i+1}(x_{i+1}))
\\
& =
\Glt{1}(x_1)
\\
& =
b_i(\xb),
\end{split}
\]
and
\[
\begin{split}
b_{i-1}\Gltt{i}(\xb) 
& =
b_{i-1}(x_1, \dots, x_i)
\\
& =
\Glt{1}(x_1)
\\
& =
b_i(\xb).
\end{split}
\]
\end{proof}

\begin{paragr}\label{paragr:shift_thtld}
Let now $C$ be equal to $\Thtld$. By the previous proposition, we have a
diagram
\[
\xymatrix{
F \ar[r]^\alpha & K & \Dn{0} \ar[l]_{\beta}
}
\]
of functors from $\G$ to $\Thtld$.
The functor $F$ is globular by definition, the
functor $K$ is globular by
Proposition \ref{prop:shifted_ge} and the functor $\Dn{0}$ is trivially globular.
This diagram thus lives in $\ExtGl(\Thtld)$.
Let $F_0 : \Thz \to \Thtld$ be the canonical functor. 
By the universal property of $\Thz$ (Proposition \ref{prop:prop_univ_thz}),
we obtain a diagram
\[
\xymatrix{
F_0 \ar[r]^{\alpha^{}_0} & K_0 & \Dn{0} \ar[l]_{\beta^{}_0}
}
\]
in $\Homgl(\Thz, \Thtld)$. Note that for the same reason as in Paragraph
\ref{paragr:shifted_under_thz}, this lifting is unique. But this diagram
lives in $\ExtGr(\Thtld)$. Indeed, $(\Thtld, F_0)$ is a groupoidal globular
extension by definition, $(\Thtld, K_0)$ is a groupoidal globular extension
by Proposition \ref{coro:ck_gr} and $(\Thtld, \Dn{0})$ is trivially
a groupoidal globular extension. Hence by the universal property of $\Thtld$ (Proposition
\ref{prop:prop_univ_thtld}), this diagram lifts to a unique diagram
\[
\entrymodifiers={+!!<0pt,\fontdimen22\textfont2>}
\xymatrix{
\raisebox{0ex}{$\id{\Thtld}$} \ar[r]^{\widetilde{\alpha}} &
\raisebox{.0ex}{$\widetilde{K}$} & \Dn{0}
\ar[l]_{\widetilde{\beta}}
}
\]
in $\Homglz(\Thtld, \Thtld)$. This is our desired
décalage on $\Thtld$. We will denote it by $\DThtld$.
\end{paragr}

\begin{paragr}\label{paragr:split_shift_thtld}
We will now construct a splitting to the décalage $\DThtld$. Let
\[
\rho^{}_i : \Dnt{i} \to \Dn{i}, \quad i \ge 0,
\]
be the morphism defined by the formula
\[
\rho^{}_i = \big(\Tht[0]{i}\Thk{0}, \dots, \Tht[i-1]{i}\Thk{i-1}, \Thk{i}\big).
\]
This morphism is not natural in $i$. For instance, the square
\[
\entrymodifiers={+!!<0pt,\fontdimen22\textfont2>}
\xymatrix{
\raisebox{0ex}{$\Dnt{0}$} \ar[d]_{\Thst{1}} \ar[r]^{\rho^{}_0} & \Dn{0}^{} \ar[d]^{\Ths{1}} \\
\raisebox{0ex}{$\Dnt{1}$} \ar[r]_{\rho^{}_1} & \Dn{1}
}
\]
is not commutative. Therefore, we cannot extend formally $\rho$ to a general
globular sum. Denote by
\[
\rho^{}_{j,i} : \Dnt{j,i} \to \Dn{i},\quad i \ge j \ge 0,
\]
the composition of the canonical morphism $\Dnt{j,i} \to \Dnt{i}$ followed
by $\rho^{}_i$. If $S$ is a globular sum whose table of dimensions
is
\[ \tabdim, \]
we define
\[ 
\rho^{}_S : \widetilde{S} = \Dnt{i_1} \amalgd{i'_1} \Dnt{i'_1+1,i_2}
\amalgd{i'_2} \dots \amalgd{i'_{n-1}} \Dnt{i'_{n-1}+1,i_n} \to S
\]
by the formula
\[
\rho^{}_S = \rho^{}_{i_1} \amalgd{i'_1} \rho^{}_{i'_1+1,i_2} \amalgd{i'_2} \dots
\amalgd{i'_{n-1}} \rho^{}_{i'_{n-1}+1, i_n}.
\]

Dually, we define maps
\[
\begin{split}
r_i & :  X_i \to \Xt{i}, \quad i \ge 0, \\
r_{j,i} & :  X_i \to \Xt{j,i}, \quad i \ge j \ge 0, \\
r^{}_S & : X(S) \to \widetilde{X}(S),\quad{\text{$S$ globular sum}},
\end{split}
\]
by the formulas
\[
\begin{split}
r_i(x_i) & = \big(\Glk{0}\Glt[0]{i}(x_i), \dots, \Glk{i-1}\Glt[i-1]{i}(x_i),
\Glk{i}(x_i)\big), \\
r_{j,i}(x_i) & = \big(\Glk{j}\Glt[j]{i}(x_i), \dots, \Glk{i-1}\Glt[i-1]{i}(x_i),
\Glk{i}(x_i)\big), \\
r^{}_S(x_{i_1}, \dots, x_{i_n}) & = \big(r_{i_1}(x_{i_1}),
r_{i'_1+1,i_2}(x_{i_2}), \dots, r_{i'_{n-1}+1,i_n}(x_{i_n})\big).
\end{split}
\]
\end{paragr}

\begin{prop}\label{prop:dthtld_split}
The $\rho^{}_S$'s are well-defined. Moreover, for every object $S$ of $\Thtld$, we
have
\[ \rho_S \widetilde{\alpha}^{}_S = \id{S}. \]
In other words, $\rho$ is a splitting of $\DThtld$.
\end{prop}

\begin{proof}
Let $i \ge 0$ and let $x_i$ be in $X_i$. To prove that $r_i(x_i)$ belongs
to $\Xt{i}$, we need to check that
\[
\Gls{l}\Glk{l-1}\Glt[l-1]{i}(x_i) = \Glt{l}\Glt{l+1}\Glk{l}\Glt[l]{i}(x_i),
\quad 1 \le l \le i.
\]
But using the identities
$\Gls{l}\Glk{l-1} = \id{X_{l-1}}$ and $\Glt{l+1}\Glk{l} = \id{X_l}$, we
get that both sides are equal to $\Glt[l-1]{i}(x_i)$. 

Let now $S$ be a globular sum whose table of dimensions is
\[ \tabdim, \]
and let
$(x_{i_1}, \dots, x_{i_n})$ be in $X(S)$.
To prove that $r^{}_S(x_{i_1}, \dots, x_{i_n})$ belongs to
$\widetilde{X}(S)$,
we need to check that
\[
\Gls[i'_l]{i_l+1}\Glk{i_l}(x_{i_l}) =
\Glt[i'_l]{i'_l+2}\Glk{i'_l+1}\Gls[i'_l + 1]{i_{l+1}}(x_{i_{l+1}}),
\quad 1 \le l \le n - 1.
\]
But this equality is equivalent to the equality
\[
\Gls[i'_l]{i_l}(x_{i_l}) =
\Glt[i'_l]{i_{l+1}}(x_{i_{l+1}})
\]
which holds by definition of $X(S)$.

Let us now prove that $r^{}_S$ is a section of $a_S$.
We easily check that $r_i$ is a section of $a_i$:
\[
\begin{split}
a_ir_i(x_i) 
& = \big(\Glk{0}\Glt[0]{i}(x_i), \dots, \Glk{i-1}\Glt[i-1]{i}(x_i),
\Glk{i}(x_i)\big) \\
& = \Gls{i+1}\Glk{i}(x_i) \\
& = x_i.
\end{split}
\]
More generally, if
\[ a_{j,i} : \Xt{j,i} \to X_i, \quad i \ge j \ge 0, \]
is defined by the formula
\[ a_{j,i}(x_{j+1}, \dots, x_{i+1}) = \Gls{i+1}(x_{i+1}), \]
the same calculation shows that $r_{j,i}$ is a section of $a_{j,i}$.

Let $\at^{}_S = X(\widetilde{\alpha}^{}_S)$
and let $\at'^{}_S$ be the morphism $\at^{}_S$ viewed
as a morphism
\[ X_{i_1} \fibrx{i'_1} \dots \fibrx{i'_{n-1}} X_{i_n} \to \Xt{i_1}
\fibrxt{i'_1} \dots  \fibrxt{i'_{n-1}} \Xt{i_n}.
\]
By definition, we have
\[
\at'^{}_S = r_{i_1} \times_{r_{i'_1}} \dots \times_{r_{i'_{n-1}}} r_{i_n}.
\]
Let $d$ be the canonical isomorphism
\[
\Xt{i_1} \fibrx{i'_1} \Xt{i'_1+1,i_2} \fibrx{i'_2} \dots 
\fibrx{i'_{n-1}} \Xt{i'_{n-1}+1, i_n} \to
\Xt{i_1} \fibrxt{i'_1} \dots  \fibrxt{i'_{n-1}} \Xt{i_n}.
\]
We recall that
\[
\begin{split}
\MoveEqLeft
d\big(x^1_1, \dots, x^1_{i_1 + 1}, x^2_{i'_1+2}, \dots, x^2_{i_2 + 1},
\dots, x^n_{i'_{n-1} + 2}, \dots, x^n_{i_n + 1}\big) \\
& = \big(x^1_1, \dots, x^1_{i_1 + 1}, x^2_1, \dots, x^2_{i_2 + 1},
\dots, x^n_1, \dots, x^n_{i_n + 1}\big),
\end{split}
\]
where the
\[ \text{$x^l_j$}, \quad 2 \le l \le n, \quad 1 \le j \le i'_l+1, \]
are defined by formulas given in Paragraph \ref{paragr:iso_can}.

We thus have
\[
\begin{split}
\MoveEqLeft
\at^{}_S\big(x^1_1, \dots, x^1_{i_1 + 1}, x^2_{i'_1+2}, \dots, x^2_{i_2 + 1},
\dots, x^n_{i'_{n-1} + 2}, \dots, x^n_{i_n + 1}\big) \\
& =
\at'_Sd\big(x^1_1, \dots, x^1_{i_1 + 1}, x^2_{i'_1+2}, \dots, x^2_{i_2 + 1},
\dots, x^n_{i'_{n-1} + 2}, \dots, x^n_{i_n + 1}\big) \\
& = \at'_S\big(x^1_1, \dots, x^1_{i_1 + 1}, x^2_1, \dots, x^2_{i_2 + 1},
\dots, x^n_1, \dots, x^n_{i_n + 1}\big) \\
& = \big(a_{i_1}(x^1_1, \dots, x^1_{i_1 + 1}), \dots,
a_{i_n}(x^n_1, \dots, x^n_{i_n + 1})\big) \\
& = \big(\Gls{i_1+1}(x^1_{i_1+1}), \dots, \Gls{i_n+1}(x^n_{i_n+1})\big),
\end{split}
\]
hence the equality
\[
\at^{}_S = a_{i_1} \fibrx{i'_1} a_{i'_1+1,i_2} \fibrx{i'_2} \dots
\fibrx{i'_{n-1}} a_{i'_{n-1}+1, i_n}.
\]

We can now compute $\at^{}_Sr^{}_S$:
\[
\begin{split}
\MoveEqLeft
\at^{}_Sr^{}_S\big(x_{i_1}, \dots, x_{i_n}\big) \\
& =
\at^{}_S\big(r_{i_1}(x_{i_1}), r_{i'_1+1,i_2}(x_{i_2}),
\dots, r_{i'_{n-1}+1, i_n}(x_{i_n})\big) \\
& =
\big(a_{i_1}r_{i_1}(x_{i_1}), a_{i'_1+1, i_2}r_{i'_1+1,i_2}(x_{i_2}),
\dots, a_{i'_{n-1}+1,i_n}r_{i'_{n-1}+1,i_n}(x_{i_n})\big) \\
& =
\big(x_{i_1}, \dots, x_{i_n}\big),
\end{split}
\]
where the last equality follows from the fact that $r_{j,i}$ is a section of
$a_{j,i}$. We thus have shown that $r^{}_S$ is a section of $\at^{}_S$.
\end{proof}

\newcommand\Deltat{\widetilde{\mathrm{\Delta}}}
\newcommand\Deltan[1]{\Delta_{#1}}
\newcommand\Deltatn[1]{\widetilde{\Delta}_{#1}}
\newcommand{\Thtldn}[1]{\widetilde{\Theta}_{#1}}
\newcommand\Thnd\nabla
\newcommand\Thkd\kappa
\newcommand\Thwd\Omega
\newcommand\Thndt{\widetilde{\nabla}}
\newcommand\Thkdt{\widetilde{\kappa}}
\newcommand\Thwdt{\widetilde{\Omega}}
\newcommand\Deltatz{\widetilde{\Delta}_0}
\begin{rem}
The category $\Thtld$ has been defined by a universal property related to
the notion of strict \oo-groupoid. For each $n \ge 1$, we can define a category
$\Thtldn{n}$ enjoying a similar universal property with respect to the
notion of strict $n$-groupoid. The category $\Thtldn{n}$ can be seen as the
full subcategory of $\Thtld$ whose objects are globular sums of dimension
at most $n$, i.e., globular sums
\[
\Dn{i_1} \amalgd{i'_1} \dots \amalgd{i'_{n-1}} \Dn{i_n},
\]
with $i_k \le n$ for all $k$ such that $1 \le k \le n$. Let us denote by $i_n$
the inclusion functor $\Thtldn{n} \to \Thtld$. This functor admits a left
adjoint $p_n : \Thtld \to \Thtldn{n}$ which truncates globular sums in
dimension $n$, i.e., which sends the globular sum
\[
\Dn{i_1} \amalgd{i'_1} \dots \amalgd{i'_{n-1}} \Dn{i_n}
\]
to the (possibly degenerated) globular sum
\[
\Dn{j_1} \amalgd{j'_1} \dots \amalgd{j'_{n-1}} \Dn{j_n},
\]
where
\[
\begin{split}
 j_k = \min(i_k, n), & \quad 1 \le k \le n,\\
 j'_k = \min(i'_k, n), & \quad 1 \le k \le n-1.
\end{split}
\]
Note that we have $p_ni_n = \id{\Thtld}$.

The décalage
\[
\DThtld =
\entrymodifiers={+!!<0pt,\fontdimen22\textfont2>}
\xymatrix{
\raisebox{0ex}{$\id{\Thtld}$} \ar[r]^{\widetilde{\alpha}} &
\raisebox{.0ex}{$\widetilde{K}$} & \Dn{0}
\ar[l]_{\widetilde{\beta}}
}
\]
induces a décalage
\[
\DThtldn{n} =
\entrymodifiers={+!!<0pt,\fontdimen22\textfont2>}
\xymatrix{
\raisebox{0ex}{$\id{\Thtldn{n}}$} \ar[r]^{\widetilde{\alpha}_n} &
\raisebox{.0ex}{$\widetilde{K}_n$} & \Dn{0}
\ar[l]_{\widetilde{\beta}_n}
}
\]
on $\Thtldn{n}$,
defined by
\[ 
\widetilde{K}_n = p_n\widetilde{K}i_n,\quad
\widetilde{\alpha}_n = p_n\ast\widetilde{\alpha}\ast i_n\quad
\text{and}\quad
\widetilde{\beta}_n = p_n\ast\widetilde{\beta}\ast i_n.
\]
Moreover, every splitting of $\DThtld$ induces a splitting of $\DThtldn{n}$.
Note that the inclusion functor $i_n : \Thtldn{n} \to \Thtld$ is \emph{not}
a morphism of décalages.

For each $n \ge 1$, the category $\Thtldn{n}$ is canonically isomorphic to
the full subcategory of the category of strict $n$-groupoids whose objects
are free strict $n$-groupoids on globular pasting schemes of dimension
at most $n$. In particular, $\Thtldn{1}$ is canonically isomorphic to the
category $\Deltat$ defined as follows:
the objects of $\Deltat$ are the sets
\[ \Deltan{n} = \{0, \dots, n\}, \quad n \ge 0, \]
and its morphisms are \emph{all} the applications between these sets.

Let us now try to understand the induced décalage on $\Deltat = \Thtldn{1}$.
The functor $\widetilde{K}_1$ sends $\Deltan{n}$ to
$\Deltan{n+1}$ and we thus set $\Deltatn{n} = \Deltan{n+1}$.
The functor $p_1 : \Thtld \to \Deltat$ sends the morphisms
\[
\begin{split}
\Thn[0]{1} & : \Dn{1} \to \Dn{1} \amalgd{0} \Dn{1},\\
\Thk{0} & : \Dn{1} \to \Dn{0},\\
\Thw[0]{1} & : \Dn{1} \to \Dn{1},\\
\end{split}
\]
to the morphisms
\[
\begin{split}
\Thnd & : \Deltan{1} \to \Deltan{1} \amalg_{\Deltan{0}} \Deltan{1} =
\Deltan{2}, \\
\Thkd & : \Deltan{1} \to \Deltan{0}, \\
\Thwd & : \Deltan{1} \to \Deltan{1},
\end{split}
\]
defined by
\[
\begin{split}
\Thnd & : 0 \mapsto 0, \quad 1 \mapsto 2,\\
\Thkd & : 0 \mapsto 0, \quad 1 \mapsto 0,\\
\Thwd & : 0 \mapsto 1, \quad 1 \mapsto 0.
\end{split}
\]
In the same way, the morphisms
\[
\begin{split}
\Thnt[0]{1} & : \Dnt{1} = \Dn{1} \amalgd{0} \Dn{2} \to \Dnt{1} \amalgdt{0}
\Dnt{1} = \Dn{1} \amalgd{0} \Dn{2} \amalgd{1} \amalgd{2},\\
\Thkt{0} & : \Dnt{1} = \Dn{1} \amalgd{0} \Dn{2}\to \Dnt{0} = \Dn{1},\\
\Thwt[0]{1} & : \Dnt{1} = \Dn{1} \amalgd{0} \Dn{2} \to \Dnt{1} = \Dn{1} \amalgd{0} \Dn{2},\\
\end{split}
\]
are sent to the morphisms
\[
\begin{split}
\Thndt & : \Deltatn{1} = \Deltan{2} \to \Deltatn{1} \amalg_{\Deltatn{0}}
\Deltatn{1} =
\Deltan{3}, \\
\Thkdt & : \Deltatn{1} = \Deltan{2} \to \Deltatn{0} = \Deltan{1}, \\
\Thwdt & : \Deltatn{1} = \Deltan{2} \to \Deltatn{1} = \Deltan{2},
\end{split}
\]
defined by
\[
\begin{split}
\Thndt & : 0 \mapsto 0, \quad 1 \mapsto 2, \quad 2 \mapsto 3,\\
\Thkdt & : 0 \mapsto 0, \quad 1 \mapsto 0, \quad 2 \mapsto 1,\\
\Thwdt & : 0 \mapsto 1, \quad 1 \mapsto 0, \quad 2 \mapsto 2.
\end{split}
\]
Let $D$ be the endofunctor of $\Deltat$ defined by 
\[ D(\Deltan{n}) = \Deltatn{n} = \Deltan{n+1} \]
for every $n \ge 0$, and by
\[
D(\varphi)(k) =
\begin{cases}
\varphi(k), & 0 \le k \le m, \\
n + 1, & k = m + 1,
\end{cases}
\]
for every morphism $\varphi : \Deltan{m} \to \Deltan{n}$ of $\Deltat$.
We have
\[
 \Thndt = D(\Thnd), \quad
 \Thkdt = D(\Thkd) \quad\text{and}\quad
 \Thwdt = D(\Thwd).
\]
Thus the functors $\widetilde{K}_1$ and $D$ agree on objects and on the
morphisms $\Thnd$, $\Thkd$ and $\Thwd$. The universal property of
$\Deltat = \Thtldn{1}$ then implies that $\widetilde{K}_1 = D$.
One can show in a similar way that the natural transformations
$\widetilde{\alpha}_1 : \id{\Deltat} \to
\widetilde{K}_1$ and $\widetilde{\beta}_1 : \Deltan{0} \to
\widetilde{K}_1$ are induced by the applications
\[
\begin{split}
\Deltan{n} & \to \Deltan{n+1}\\
k & \mapsto k
\end{split}
\qquad\text{and}\qquad
\begin{split}
\Deltan{0} & \to \Deltan{n+1}\\
0 & \mapsto n+1.
\end{split}
\]

Note that this décalage restricts to the subcategory $\mathrm{\Delta}$ of
$\Deltat$ whose objects are the $\Deltan{n}$'s and whose morphisms are
order-preserving maps. The induced décalage on $\mathrm{\Delta}$ is precisely
the one defined in Example 3.14 of \cite{MaltsiTheta}.
\end{rem}

\section{$\Thtld$ is a test category}

\begin{prop}
The object $\Dn{0}$ is terminal in $\Thtld$.
\end{prop}

\begin{proof}
This is an immediate consequence of Proposition \ref{prop:desc_thtld}.
\end{proof}

\begin{prop}\label{prop:sep_int}
$(\Dn{1}, \Ths{1}, \Tht{1})$ is a separating interval on $\pref{\Thtld}$.
\end{prop}

\begin{proof}
We need to show that the equalizer  of $\Ths{1}, \Tht{1} : \Dn{0} \to
\Dn{1}$ in $\pref{\Thtld}$ is the initial presheaf, i.e., that there does not
exist an object $S$ in $\Thtld$ such that the diagram
\[
\xymatrix{
S \ar[r] & \Dn{0} \ar@<.6ex>[r]^-{\Ths{1}} \ar@<-.6ex>[r]_-{\Tht{1}} &
\Dn{1}
}
\]
is commutative. Suppose that such an $S$ exists. By precomposing with a morphism from
$\Dn{0}$, we can assume that $S$ is $\Dn{0}$. Since by the previous
proposition, $\Dn{0}$ is a terminal object, that would imply that $\Ths{1}$
and $\Tht{1}$ are equal. By the universal property of $\Thtld$ (Proposition
\ref{prop:mod_thtld}), that would mean that $\Gls{1}$ and $\Glt{1}$ are
equal for every strict \oo-groupoid. This is obviously false.
\end{proof}

\begin{thm}\label{thm:thtld_test}
The category $\Thtld$ is a strict test category.
\end{thm}

\begin{proof}
By the previous proposition, $(\Dn{1}, \Ths{1}, \Tht{1})$ is a separating interval
on $\pref{\Thtld}$. Moreover, since $\Dn{1}$ is a representable presheaf, it is
aspherical. Furthermore, by Paragraph \ref{paragr:shift_thtld} and
Proposition \ref{prop:dthtld_split}, $\Thtld$ admits a splittable décalage.
Hence the result by Proposition \ref{prop:crit_test}.
\end{proof}

\begin{coro}\label{coro:thtld_mcs}
The pair $(\pref{\Thtld}, \W_{\Thtld})$ is endowed with a structure of model
category whose cofibrations are the monomorphisms. This model category
structure is cofibrantly generated, proper and the weak equivalences are
stable by binary products.

Moreover, the homotopy category $\Ho{\Thtld}$ of $\pref{\Thtld}$
is canonically equivalent to the homotopy category $\Hot$.
\end{coro}

\begin{proof}
This follows from the previous theorem by Theorem \ref{thm:test_cm}.
\end{proof}

\begin{paragr}
A similar proof (using the very same calculations) shows analogous
results for the category $\Theta$. Indeed, let $(C, F_0)$ be a categorical
globular extension. The definitions of the $\Thn[j]{i}$'s and $\Thk{i}$'s of
Paragraph \ref{paragr:def_thn}
still make sense. Moreover, by Propositions \ref{prop:thnt}, \ref{prop:ass},
\ref{prop:exc}, \ref{prop:thkt}, \ref{prop:lrun} and \ref{prop:fun},
the twisted globular extension $(C, K_0)$ under $\Thz$, endowed with these morphisms, is a
categorical globular extension. 
By this result and the universal property of $\Th$, 
we can construct a décalage $\DTh$ on $\Th$ as we did in Paragraph
\ref{paragr:shift_thtld} for $\Thtld$. The definition of the
$\rho^{}_S$'s of Paragraph \ref{paragr:split_shift_thtld} still makes sense
and the proof of Proposition \ref{prop:dthtld_split} applies and shows that
$\rho$ is a splitting of $\DTh$. Moreover, the proof of Proposition
\ref{prop:sep_int} shows that $(\Dn{1}, \Ths{1}, \Tht{1})$ is a separating
interval on $\pref{\Th}$. We hence obtain by Theorem \ref{prop:crit_test} that $\Th$
is a strict test category. In particular, $\pref{\Th}$ is endowed with a model
category structure as in Corollary \ref{coro:thtld_mcs}.
One can show that the décalage $\DTh$ and its splitting are the same as those
constructed in \cite{MaltsiTheta}. 

Moreover, since the décalages $\DTh$ and $\DThtld$ are defined in a uniform way,
the canonical functor $i : \Th \to \Thtld$ (obtained by the universal property of
$\Th$) induces a morphism of décalages. 
Proposition \ref{prop:shift_asp} thus implies the following theorem.
\end{paragr}

\begin{thm}\label{thm}
The canonical functor $i : \Th \to \Thtld$ is aspherical.
\end{thm}

\begin{coro}
Let $i^\ast : \pref{\Thtld} \to \pref{\Th}$ be the restriction functor and
let $i_\ast$ be its right adjoint. Then $(i^\ast, i_\ast)$ is a Quillen
equivalence.
\end{coro}

\begin{proof}
This follows from the previous theorem by Proposition \ref{prop:asp}.
\end{proof}

\begin{paragr}
If $I$ is a subset of $\{l, r, f\}$, let us denote by $\Th_I$ the universal
precategorical globular extension satisfying Axioms $\Assx$ and $\Excx$,
plus Axiom $\Lunx$ (respectively $\Runx$, respectively $\Funx$) if $l$
(respectively $r$, respectively $f$) belongs to $I$. In particular, we have 
$\Th = \Th_{l,r,f}$.

In the same way, if $J$ is a subset of $\{l, r, f, \tilde{l}, \tilde{r}\}$,
let us denote by $\Thtld_J$ the universal pregroupoidal globular extension
satisfying the same axioms as $\Th_{J\cap\{l,r,f\}}$ plus Axiom $\LInvx$
(respectively $\RInvx$) if $\tilde{l}$ (respectively $\tilde{r}$) belongs to
$J$. In particular, we have $\Thtld = \Thtld_{l, r, f, \tilde{l}, \tilde{r}}$.

A closer look at the calculations of the previous sections reveals
that 
\[
\xymatrix@C=1pc@R=2pc{
& \Th_{lr} \ar[rr] \ar[dr] & & \Thtld_{lr\tilde{r}} \ar[rr] \ar[dr] &  &
\Thtld_{lr\tilde{l}\tilde{r}} \ar[rd] \\
\Th_r \ar[ru] \ar[rd] & & \Th = \Th_{lrf} \ar[rr] & &
\Thtld_{lrf\tilde{r}} \ar[rr] & &
\Thtld = \Thtld_{lrf\tilde{l}\tilde{r}} \\
    & \Th_{rf} \ar[ru] \\
}
\]
is a diagram of strict test categories and aspherical functors.

By duality, the diagram obtained by exchanging $l$ and $r$, and
$\tilde{l}$ and $\tilde{r}$, is also a diagram of strict test categories and
aspherical functors.
\end{paragr}

\bibliographystyle{amsplain}
\bibliography{biblio}

\end{document}